\documentclass[11pt]{article}
\usepackage[margin = 1in]{geometry}

\usepackage{amsmath,amssymb,amsthm}
\usepackage{multirow}
\usepackage{graphicx}
\usepackage{longtable}
\usepackage{color}
\usepackage{changepage}
\usepackage{circuitikz}
\usepackage{gensymb}
\usepackage{algpseudocode}
\usepackage[ruled,vlined]{algorithm2e}
\usepackage{cancel}
\usepackage{colortbl}
\usepackage{caption}
\usepackage{mathrsfs}

\usepackage{subcaption}
\usepackage{array}
\usepackage{tikz}
\usetikzlibrary{positioning, arrows.meta}
\usepackage{tikz}
\usepackage{tikz-cd}
\usepackage{diagbox}
\usepackage{authblk}
\tikzstyle{box1} = [rectangle, rounded corners, minimum width=4cm, minimum height=1.1cm,text centered, draw=black, fill=blue!5]
\tikzstyle{box2} = [rectangle, rounded corners, minimum width=4.5cm, minimum height=1.1cm,text centered, draw=black, fill=red!5]
\tikzstyle{box3} = [rectangle, rounded corners, minimum width=4.5cm, minimum height=1.1cm,text centered, draw=black, fill=green!5]
\tikzstyle{arrow} = [thick,->,>=stealth]

\usepackage{blindtext}
\usepackage{titlesec}
\usepackage[colorlinks=true, allcolors=blue]{hyperref}
\usepackage{cleveref}

\newtheorem{theorem}{Theorem}[section]
\newtheorem*{variance1}{Principle of Optimal Variance Reduction (V1)}
\newtheorem*{variance2}{Principle of Optimal Variance Reduction (V2)}
\newtheorem{proposition}{Proposition}[section]

\newtheorem{lemma}{Lemma}[section]
\theoremstyle{plain}

\usepackage{bm}

\theoremstyle{definition}
\newtheorem{definition}{Definition}[section]
\newtheorem{assumption}{Assumption}[section]

\usetikzlibrary{matrix,decorations.pathreplacing}

\pgfkeys{tikz/mymatrixenv/.style={decoration={brace},every left delimiter/.style={xshift=8pt},every right delimiter/.style={xshift=-8pt}}}
\pgfkeys{tikz/mymatrix/.style={matrix of math nodes,nodes in empty cells,left delimiter={[},right delimiter={]},inner sep=1pt,outer sep=1.5pt,column sep=2pt,row sep=2pt,nodes={minimum width=20pt,minimum height=10pt,anchor=center,inner sep=0pt,outer sep=0pt}}}
\pgfkeys{tikz/mymatrixbrace/.style={decorate,thick}}
\newcommand*\mymatrixbraceright[4][m]{
    \draw[mymatrixbrace] (#1.west|-#1-#3-1.south west) -- node[left=2pt] {#4} (#1.west|-#1-#2-1.north west);
}

\DeclareMathOperator*{\argmin}{arg\,min}

\definecolor{blizzardblue}{rgb}{0.67, 0.9, 0.93}
\definecolor{babyblueeyes}{rgb}{0.63, 0.79, 0.95}
\definecolor{darklavender}{rgb}{0.45, 0.31, 0.59}
\definecolor{darkmagenta}{rgb}{0.55, 0.0, 0.55}
\definecolor{deepskyblue}{rgb}{0.0, 0.75, 1.0}
\definecolor{frenchblue}{rgb}{0.0, 0.45, 0.73}

\title{Invariant Measures for Data-Driven Dynamical System Identification: Analysis and Application}
\author[1]{Jonah Botvinick-Greenhouse}

\affil[1]{\small\textit{Center for Applied Mathematics, Cornell University, Ithaca, NY 14853}}
\begin{document}
\maketitle

\begin{abstract}
 We propose a novel approach for performing dynamical system identification, based upon the comparison of simulated and observed physical invariant measures. While standard methods adopt a Lagrangian perspective by directly treating time-trajectories as inference data,  we take on an Eulerian perspective and instead seek models fitting the observed global time-invariant statistics. With this change in perspective, we gain robustness against pervasive challenges in system identification including noise, chaos, and slow sampling. In the first half of this paper, we pose the system identification task as a partial differential equation (PDE) constrained optimization problem, in which synthetic stationary solutions of the Fokker--Planck equation, obtained as fixed points of a finite-volume discretization, are compared to physical invariant measures extracted from observed trajectory data. In the latter half of the paper, we improve upon this approach in two crucial directions. First, we develop a Galerkin-inspired modification to the finite-volume surrogate model, based on data-adaptive unstructured meshes and Monte--Carlo integration, enabling the approach to efficiently scale to high-dimensional problems. Second, we leverage Takens' seminal time-delay embedding theory to introduce a critical data-dependent coordinate transformation which can guarantee unique system identifiability from the invariant measure alone. This contribution resolves a major challenge of system identification through invariant measures, as systems exhibiting distinct transient behaviors may still share the same time-invariant statistics in their state-coordinates. Throughout, we present comprehensive numerical tests which highlight the effectiveness of our approach on a variety of challenging system identification tasks.
\end{abstract}


\section{Introduction}
Constructing data-driven models of dynamical systems is a fundamental challenge in many physics and engineering applications, including fluid flow surrogate modeling \cite{lee2021parameterized}, ion thruster model calibration \cite{greve2019data}, gravitational wave parameter estimation \cite{RevModPhys.94.025001},  and weather prediction  \cite{bochenek2022machine}. Such models can facilitate insights into the underlying physics of complex systems and serve as valuable tools for performing state prediction and control \cite{MONTANS2019845}. The task of system identification is typically formulated as an optimization procedure, where empirical time-series measurements are used to calibrate unknown parameters in a family of candidate models, often expressed as differential equations~\cite{baake1992fitting,michalik2009incremental}. Sparse sampling, noisy measurements, partial observations, and chaotic dynamics represent common challenges in system identification for discovering physically informative models.


Popular approaches for modeling dynamical trajectories typically adopt a \textit{Lagrangian perspective} and seek a pointwise matching with either the observed data or its approximate state derivatives. The class of methods termed Sparse Identification of Nonlinear Dynamics (SINDy) computes divided differences along observed trajectories and uses sparse regression to fit a linearly-parameterized vector field consisting of candidate terms from low-degree function bases, e.g., global polynomial or Fourier bases \cite{brunton2016discovering}. Shooting-type methods and neural differential equations instead attempt to recover a vector field which simulates trajectories pointwisely matching the observed data \cite{baake1992fitting,michalik2009incremental,linot2023stabilized}. In their simplest form, these approaches parameterize the unknown vector field $v = v_{\theta}$ by $\theta \in \Theta\subseteq \mathbb{R}^p$ and solve the optimization problem, e.g.,
\begin{equation}\label{eq:optimization}
    \min_{\theta\in \Theta}\mathcal{J}(\theta),\qquad \mathcal{J}(\theta):= \frac{1}{N}\sum_{i=0}^{N-1} |x^*(t_i) - x_{\theta}(t_i)|^2,
\end{equation}
where $\{x^*(t_i)\}_{i=0}^{N-1}$ is the observed time-series data and $x_{\theta}(\cdot)$ solves the initial value problem $\dot{x}_{\theta} = v_{\theta}(x_{\theta})$ with $x_{\theta}(t_0) = x^*(t_0)$. When the observed trajectory $\{x^*(t_i)\}_{i=0}^{N-1}$ is long, variations based upon multiple-shooting, which partition the loss \eqref{eq:optimization} along many sub-trajectories, have been shown to improve the optimization landscape \cite{peifer2007parameter}.

While the approaches described above can lead to effective model discovery in many scenarios, their use is also limited by relatively strong data requirements. Noisy measurements and slow sampling of the observed data prevent SINDy from accurately estimating state derivatives, which directly affects the quality of the velocity reconstruction. Approaches that minimize the pointwise errors between simulated and observerved trajectories following \eqref{eq:optimization} are similarly prone to overfitting noise in the data. When the underlying dynamics are chaotic, it is impossible to distinguish between modeling errors and intrinsic dynamic instability, which is yet another source of difficulty when directly using trajectories to reconstruct the velocity. Lastly, under slow sampling conditions, i.e., when $\Delta t = t_{i+1} - t_i$ is large, the optimization landscape given by \eqref{eq:optimization} becomes intractable; see Figure~\ref{fig:initcompare}. All of these phenomena share the nature that small perturbations to the trajectory data can have a large impact on the reconstructed velocity. 

\begin{figure*}[h!]
    \centering
    \includegraphics[width = \textwidth]{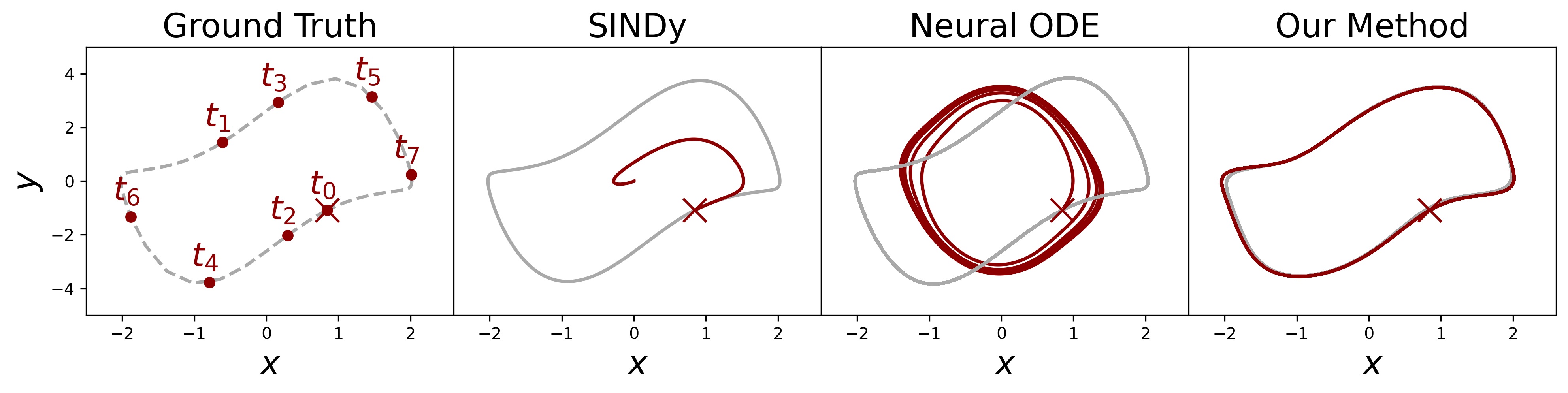}
    \caption{Comparison with the SINDy and the Neural ODE frameworks for reconstructing the velocity from slowly sampled observations. While SINDy and Neural ODE can only reconstruct an accurate model from a quickly sampled trajectory, our approach is robust to slowly sampled data. See Section \ref{sec:numerics} for additional experiment details.}
    \label{fig:initcompare}
\end{figure*}

In contrast with the Lagrangian approach to modeling dynamics, we adopt an \textit{Eulerian perspective} in which models are constructed to match the global time-invariant statistics of observed trajectories \cite{greve2019data,yang2023optimal}. Rather than directly using the observed data $\{x^*(t_i)\}_{i=1}^N$ to reconstruct the velocity, we consider the \textit{occupation measure} $\rho^*$, where for each measurable set $B$,
\begin{equation}\label{eq:occupation}
    \rho^*(B):= \frac{1}{N}\sum_{i=1}^{N}\chi_B(x^*(t_i)),\qquad  \qquad\chi_B(x)=\begin{cases} 1& x\in B\\
    0 & x\not\in B
    \end{cases}.
\end{equation}
Under certain assumptions, the  occupation measure provides an approximation to the system's \textit{physical invariant measure}; see Section \ref{subsec:PM}. Notably, invariant measures can be well-approximated even under slow sampling conditions, and they are resilient to measurement noise \cite{jiang2023training}. Moreover, while nearby trajectories of a chaotic system diverge exponentially quickly, the corresponding occupation measures \eqref{eq:occupation} still converge to the same invariant measure under mild assumptions. Thus,~\eqref{eq:occupation} is robust to many of the data challenges which cause instability  in Lagrangian approaches for parameter identification; e.g.,~\eqref{eq:optimization}. Motivated by these observations, we seek to recover the parameterized velocity $v_{\theta}$ by solving
\begin{equation}\label{eq:opt2}
    \min_{\theta \in \Theta} \mathcal{J}(\theta), \qquad \mathcal{J}(\theta) = \mathcal{D}(\rho^*,\rho_{\varepsilon}(v_{\theta})).
\end{equation}
The formulation~\eqref{eq:opt2} represents an inverse data-matching problem, in which $\mathcal{D}$ denotes a metric or divergence on the space of probability measures and $\rho_{\varepsilon}(v_{\theta})$ is an approximation to the invariant measure of the dynamical system $v_{\theta}$, given some regularization parameter $\varepsilon > 0$; see Section \ref{subsec:FWD}. Whereas $\theta\mapsto \{x_{\theta}(t_i)\}_{i=1}^N$ was the forward model in the Lagrangian formulation \eqref{eq:optimization}, our new forward model is now given by $\theta \mapsto \rho_{\varepsilon}(v_{\theta})$.

Although one could approximate $\rho_{\varepsilon}(v_{\theta})$ by numerically integrating a trajectory and binning the observed states to a histogram \cite{greve2019data}, this approach does not permit simple differentiation of the resulting measure with respect to the parameter $\theta$. When the size of $\theta\in \mathbb{R}^p$ is large, it is practical to use gradient-based optimization methods for solving the optimization problem~\eqref{eq:opt2}, and one has to compute the gradient $\nabla_{\theta} \mathcal{J}$. Following \cite{yang2023optimal}, we instead view $\rho_{\varepsilon}(v_{\theta})$ as the dominant eigenvector of a regularized Markov matrix $M_{\varepsilon}(v_{\theta})\in \mathbb{R}^{n\times n}$ originating from an upwind finite volume discretization of a Fokker--Planck equation. Thus, the formulation~\eqref{eq:opt2} is translated into a \textit{PDE-constrained optimization} problem, in which the gradient $\nabla_{\theta}\mathcal{J}$ can be efficiently computed using the adjoint-state method or automatic differentiation.

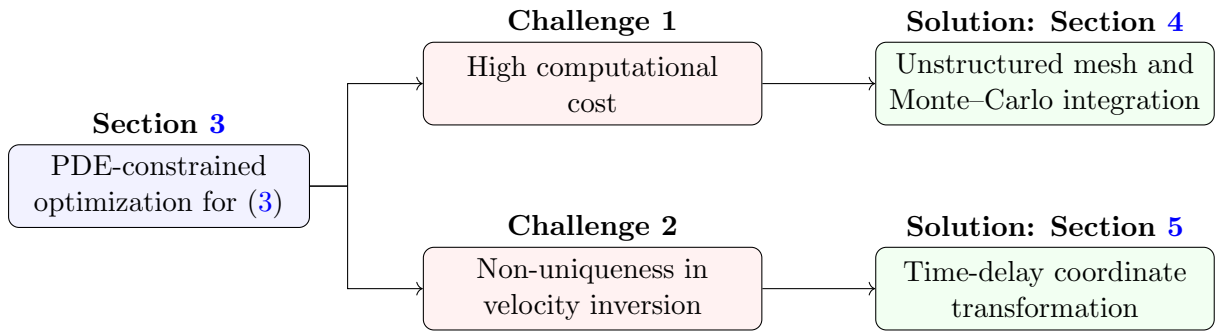
\begin{figure}[h!]
    \begin{tikzpicture}[node distance=2cm]
    \node[align=center] (textblock) {\textbf{Section \ref{sec:PDE_IM}}};
\node[draw, align = center,below =.01cm of textblock] (start) [box1] {PDE-constrained \\ optimization  for \eqref{eq:opt2}};
    \node[draw, align=center, above right=.25cm and 1.5cm of start] (red1) [box2] {High computational\\
    cost};
    \node[draw, align=center, below right=.25cm and 1.5cm of start] (red2) [box2] {Non-uniqueness in \\
    velocity inversion};
\node[draw, align=center, right=1.5cm of red1] (green1) [box3] {Unstructured mesh and\\
Monte--Carlo integration};
    \node[draw, align=center, right=1.5cm of red2] (green2) [box3] {Time-delay coordinate\\
    transformation};
        \node[right=0.5cm of start, draw=none] (invisible) {};

  \node[align=center, above = .01cm of green1] (title2) {\textbf{Solution: Section \ref{sec:mesh}}};
  \node[align=center, above = .01cm of green2] (title3) {\textbf{Solution: Section \ref{sec:delay_IM}}};
    \node[align=center, above = -0.05cm of red1] (c1) {\textbf{Challenge 1}};
      \node[align=center, above = -0.05cm of red2] (c2) {\textbf{Challenge 2}};
\draw (invisible.west) -| (start.east);
\draw (invisible.west) -| (start.east);

\draw[->] (invisible.west) |- (red1.west);
    \draw[->] (invisible.west) |- (red2.west);
    \draw[->] (red1.east) -- (green1.west);
    \draw[->] (red2.east) -- (green2.west);

    \end{tikzpicture}
    \caption{ Flowchart describing the paper's main sections and techniques.}
    \label{fig:enter-label}
\end{figure}
In Section \ref{sec:PDE_IM}, we study invariant measure-based system identification with a large-scale parameter space applied to real data. We exhibit competitive performance of our proposed approach~\eqref{eq:opt2} against standard methods for performing the velocity reconstruction, e.g., \eqref{eq:optimization}, under challenging data scenarios; see Figure \ref{fig:initcompare}. We also showcase the effectiveness of our method on real-world Hall-effect thruster dynamics, and we use the reconstructed Fokker--Planck dynamics to perform uncertainty quantification for the predicted dynamics. While the framework presented in Section~\ref{sec:PDE_IM} is largely successful, we further improve the approach in two key directions.

First, in Section \ref{sec:mesh} we examine alternative approaches to the construction of the Markov matrix $M_{\varepsilon}(v_{\theta})$ which can accelerate computation of the forward model. While obtaining $M_{\varepsilon}(v_{\theta})$ based on the discretization of the Fokker--Planck differential operator on a uniformly-spaced mesh is efficient for low-dimensional dynamical systems, the approach becomes infeasible when the dimension of the state exceeds 4. In practice, the intrinsic dimension of a system's attractor may be significantly smaller than the ambient dimension of the state space in which the system evolves \cite{sauer1991embedology,farmer1983dimension}. Motivated by this phenomena, we instead use a data-adaptive unstructured mesh with cells $\{C_i\}_{i=1}^n$ concentrated on the attractor of the observed trajectory $\{x^*(t_i)\}_{i=0}^{N-1}.$ While the finite-volume method (FVM) can be viewed as one approach for approximating the so-called \textit{Perron--Frobenius Operator} (PFO) \cite{norton2018numerical}, in Section~\ref{sec:mesh} we instead obtain $M_{\varepsilon}$ as a regularized Galerkin projection of the PFO onto the indicator functions supported on the cells~\cite{Stefan, li1976finite}, which we approximate via Monte--Carlo integration. Based on our error analysis we propose a method for cell construction based upon the principle of \textit{optimal variance reduction}, and we prove the convergence of our approach to the true Perron--Frobenius operator as all discretization and regularization parameters are refined; see Theorem~\ref{thm:convergence}. Due to the data-efficiency of our unstructured mesh approach, it now becomes computationally feasible to compare entire Markov matrices during optimization, rather than simply the dominant eigenvectors, as in~\eqref{eq:opt2}. We conclude Section \ref{sec:mesh} by considering a variant of \eqref{eq:opt2}, and we demonstrate its efficiency applied to the 30-dimensional Lorenz-96 model \cite{karimi2010extensive}.

\begin{figure*}[h!]
    \centering
    \includegraphics[width = .9\textwidth]{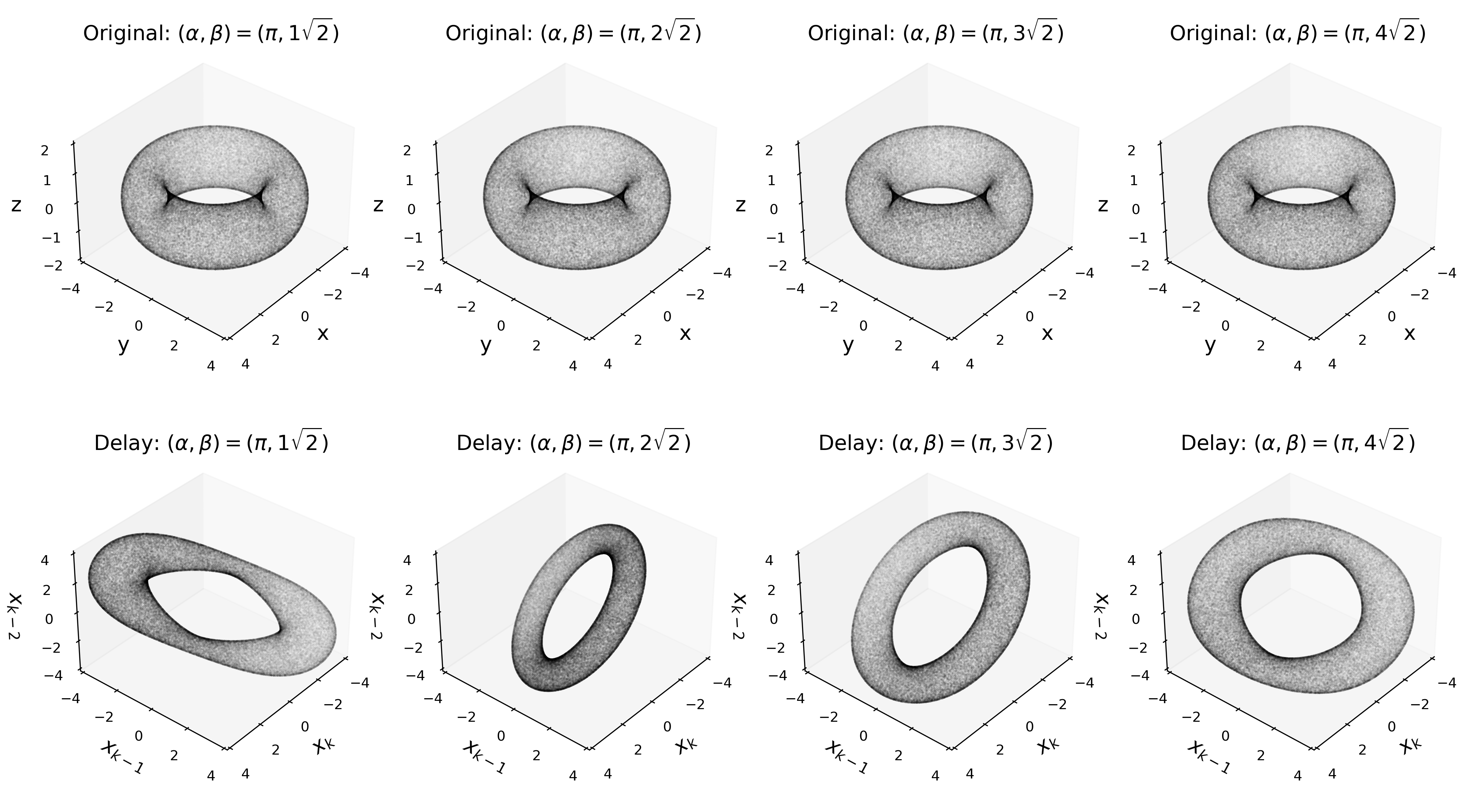}
    \vspace{-.5cm}
    \caption{Delay-coordinate invariant measures improve identifiability of the torus rotation $ T_{\alpha,\beta}(z_1,z_2)= (z_1+\alpha,z_2+\beta) \pmod{ 1}.$ While different choices of $(\alpha,\beta)$ all lead to the same state-coordinate invariant measure (top row), the systems can be distinguished by their delay-invariant measures (bottom row).  }
    \label{fig:torus}
\end{figure*}
Second, the inverse problem \eqref{eq:opt2} suffers from a lack of uniqueness, as systems with different transient dynamics may still exhibit identical asymptotic statistical behavior. Motivated by the delay-coordinate transformation from Takens’ seminal embedding theory~\cite{sauer1991embedology,Takens1981DetectingSA}, we propose using \textit{invariant measures in time-delay coordinates} for dynamical system identification. As illustrated in Figure~\ref{fig:torus}, systems that share the same invariant measure in their original state-coordinates can display distinct invariant measures in time-delay coordinates, offering greater insight into their underlying dynamics. In Section \ref{sec:delay_IM}, we present two key results. First, we prove that if two dynamical systems share the same invariant measure in time-delay coordinates, they are topologically conjugate on the support of their invariant measures; see Theorem \ref{thm:1}. Second, we show that the non-uniqueness from the conjugacy relation can be eliminated by using multiple delay-coordinate invariant measures derived from different observables; see Theorem \ref{thm:2}. A summary of these results is depicted in the flowchart in Figure~\ref{fig:flowchart}. Our work is the first to provide theoretical guarantees supporting the use of delay-coordinate invariant measures for system identification.

The structure of the paper is as follows. Section \ref{sec:background} contains a review of essential background material and related work. The PDE-constrained optimization approach \eqref{eq:opt2} for reconstructing the velocity from noisy measurements is presented in Section~\ref{sec:PDE_IM}. The data-adaptive mesh reformulation, which improves upon the previous approach's computational efficiency, is then explored in Section~\ref{sec:mesh}. In Section \ref{sec:delay_IM}, we pose the inverse problem \eqref{eq:opt2} in a time-delay coordinate system, where we argue that the inverse solution is unique up to a topological conjugacy.  Conclusions then follow in Section~\ref{sec:conclusions}.

\section{Background}\label{sec:background}
\subsection{Discrete, Continuous, and Stochastic Dynamics}\label{subsec:dynamics}
A discrete-time dynamical system is a self-mapping $T:X\to X$ of a set $X$, commonly referred to as the state-space. The dynamics on $X$ are given by iterated compositions with the map $T$. That is, if one fixes an initial condition $x\in X$, then the corresponding trajectory is given by $\{T^k(x)\}_{k=1}^{\infty}$, where $k\in\mathbb{N}$ indexes time. When the state-space has smooth structure, e.g., an open subset of Euclidean space or a Riemannian manifold, one can consider the continuous-time system $\dot{x} = v(x)$, where $v$ is a vector field on $X$ with sufficient regularity to guarantee well-posedness of the initial value problem. Given an initial condition $x\in X$ its trajectory under the continuous dynamics is $\{f_t(x)\}_{t\geq 0},$ where $f_t:X\to X$ is the time-$t$ flow map of the vector field $v$. Lastly, we may consider the It\^o stochastic differential equation (SDE)
\begin{equation}\label{ES}
 dX_t = v(X_t) dt+\sigma(X_t) dW_t, \hspace{1cm} X_0 = x,
\end{equation}
where $W_t$ is a Brownian motion, $v$ is the velocity, and $\sigma$ determines the diffusion matrix $\Sigma (x) = \frac{1}{2} \sigma(x) \sigma(x)^\top.$ In this case, a single trajectory $\{X_t\}_{t\geq 0}$ of the system  is comprised of random variables following \eqref{ES}. Note that a stochastic system with $\sigma = 0$ reduces to the standard continuous time system $\dot{x} = v(x)$. Moreover, by considering the time-$t$ flow of a vector field for fixed $t > 0$, i.e., $f_t:X\to X$, one obtains a discrete-time dynamical system. 

For simplicity, we present the rest of the background material in the setting of discrete-time dynamical systems, though  there are analogous formulations for continuous-time and stochastic dynamics.
\subsection{The Perron--Frobenius Operator}\label{subsec:perron}

The Perron--Frobenius Operator (PFO) plays a central role in analyzing the statistical properties of dynamical systems. While the state-space dynamics given by a discrete-time map $T$ are in general non-linear, the PFO provides an equivalent description of these dynamics which is linear and infinite dimensional. In particular, the PFO acts on the space of probability densities over the state-space and is adjoint to the well-known Koopman operator \cite{lasota2013chaos}. The invariant densities of a dynamical system thus correspond to the fixed points of the PFO, which make it an essential tool in both our analysis of and algorithm development  for invariant measure-based system identification. 

Given a measure space $(X,\mathscr{B},\nu)$, a measurable map $T:X\to X$ is said to be \textit{non-singular} with respect to $\nu$ if for any $B\in\mathscr{B}$ which satisfies $\nu(B) = 0$ it also holds that $\nu(T^{-1}(B)) = 0.$ The PFO can now be defined as follows.
\begin{definition}[Perron--Frobenius operator]\label{def:PFO}
    Given a non-singular system  $T:X\to X$, the \textit{Perron--Frobenius operator} $\mathcal{P}:L_{\nu}^1(X)\to L_{\nu}^1(X)$ is uniquely characterized by the relationship
\begin{equation}\label{eq:ulam}
    \int_B \mathcal{P} f d\nu = \int_{T^{-1}(B)} f d\nu ,\qquad \forall B\in \mathscr{B}, \qquad \forall f\in L_{\nu}^1(X).
\end{equation}
\end{definition} 
The PFO is a  \textit{Markov operator} and thus preserves total mass and non-negativity \cite{lasota2013chaos}. That is, if $f\in L_{\nu}^1(X)$ satisfies $f \geq 0$, then $\mathcal{P}f \geq 0$ and $\int_X f d\nu = \int_X \mathcal{P}f d\nu.$ 

If one wishes to use the PFO in practice to describe either the time-evolution of a density or the invariant measures of a dynamical system, a suitable finite-dimensional discretization of the infinite-dimensional operator must be employed \cite{klus2015numerical}. In Section \ref{sec:PDE_IM} we numerically approximate the PFO using a finite-volume discretization of the Fokker--Planck equation, and in Section \ref{sec:mesh} we approximate the PFO via a Galerkin-inspired approach \cite{li1976finite}. 
\subsection{Physical Invariant Measures}\label{subsec:PM}
 Throughout, we will formulate many of our definitions and results in the abstract setting of Polish spaces. A \textit{Polish space} $X$ is a separable completely metrizable topological space. We will denote by $\mathscr{P}(X)$ the space of \textit{probability measures} on a Polish space $X$, which consists of all Borel measures that assign unit measure to $X$. Throughout our results, it will be important to reference the support of a given probability measure, which can intuitively be viewed as the region on which the measure is concentrated. 
\begin{definition}[Support of a measure]\label{def:support}
    If $X$ is a Polish space and $\mu\in \mathscr{P}(X)$ is a Borel probability measure, then the \textit{support} of $\mu$ is the unique closed set $\text{supp}(\mu)\subseteq X$ such that $\mu(\text{supp}(\mu)) = 1$, and for any other closed set $C\subseteq X$ satisfying $\mu(C) = 1$, it holds that $\text{supp}(\mu)\subseteq C$.
\end{definition}
We will frequently need to discuss how probability measures are transformed under the action of a measurable function. This concept is formalized with the notion of a pushforward measure. 
\begin{definition}[Pushforward measure]
    Let $X$ and $Y$ be Polish spaces, let $f:X\to Y$ be Borel measurable, and let $\mu\in \mathscr{P}(X).$ Then, the \textit{pushforward} of $\mu$ is the probability measure $f\# \mu \in \mathscr{P}(Y)$, defined by $(f\# \mu)(B) := \mu(f^{-1}(B))$, for all Borel sets $B\subseteq Y$. 
\end{definition}
We next define the notion of an invariant measure, which is a probability measure that remains unchanged under the pushforward of a given self-transformation $T:X\to X$ and provides a statistical description of its asymptotic behavior.
\begin{definition}[Invariant measure]
    Let $X$ be a Polish space and let $T:X\to X$ be Borel measurable. The probability measure $\mu \in \mathscr{P}(X)$ is said to be \textit{$T$-invariant} (or simply \textit{invariant} when the map $T$ is clear from context) if $T\# \mu = \mu$. 
\end{definition}
If $\mu$ admits a density $f \geq 0$ with respect to the measure $\nu$ used to define the PFO (see \eqref{eq:ulam}), i.e., $d\mu = f d\nu$, then it additionally holds that $\mathcal{P}f = f$. Moreover, the measure $\mu$ is said to be $T$-ergodic if $T^{-1}(B) = B$ implies that $\mu(B) \in \{0,1\}$, i.e., when all invariant sets trivially have either zero or full measure. Given a dynamical system $T:X\to X$, it is natural to ask which initial conditions $x\in X$ produce trajectories $\{T^k(x)\}_{k\in\mathbb{N}}$ that yield empirical measures which are asymptotic to a given invariant measure. Towards this, we now introduce the notion of a basin of attraction.
\begin{definition}[Basin of attraction]\label{def:basin}
    Let $X$ be a Polish space, and let $T:X\to X$ be Borel measurable. The \textit{basin of attraction} of a $T$-invariant measure $\mu\in \mathscr{P}(X)$ with compact support is
   \begin{equation}\label{eq:basin}
       \mathcal{B}_{\mu,T} := \Bigg\{x\in X: \lim_{N\to\infty}\frac{1}{N}\sum_{k=0}^{N-1} \phi(T^k(x)) = \int_X \phi d\mu, \, \forall \phi \in C(X) \Bigg\}.
   \end{equation} 
\end{definition}
 If the space $X$ admits a Lebesgue measure, then $\mu$ is said to be a \textit{physical measure} when the basin of attraction \eqref{eq:basin} has positive Lebesgue measure. From the perspective of experimentalists, physical measures are the relevant invariant measures which can be observed during data collection. Under certain assumptions on the space $X$ and the system $T$, the existence of physical measures is known (see \cite[Theorem~1]{Young2002}), with common examples including attracting fixed points, stable limit cycles, and various chaotic attractors, e.g., the Lorenz-63 system~\cite{luzzatto2005lorenz}.
 
\subsection{Time-Delay Embedding}\label{sec:time_delay}

The technique of time-delay embedding has revolutionized the way in which nonlinear trajectory data is studied across a diverse range of applications, including fluid mechanics \cite{yuan2021flow}, neuroscience~\cite{tajima2015untangling}, and time-series prediction \cite{young2023deep},  to name a few. Rather than observing trajectories $\{T^k(x)\}$ of the full state of a dynamical system, practitioners may only have access to time-series projections of the form $\{y(T^k(x))\}$, for some scalar-valued observable $y.$ In certain situations, the method of time-delays can provide a reconstruction of the original dynamical system, up to topological equivalence, using only this scalar time-series data. 

That is, if the \textit{embedding dimension} $m \in \mathbb{N}$ is chosen sufficiently large, then the state-coordinate trajectory $\{T^k(x)\}$ can be placed into a continuous one-to-one correspondence with the so-called delay-coordinate trajectory $\{(y(x),y(T(x)),\dots,y(T^{m-1}(x))\}$. We remark that many approaches for numerically determining a suitable embedding dimension from time-series data, such that the delay-coordinate dynamics are well-defined, have been explored \cite{cao1997practical}. The theoretical justification of delay-coordinate reconstructions was first given by Takens \cite{Takens1981DetectingSA}, and later generalized by Sauer, Yorke, and Casdagli \cite{sauer1991embedology}, which is the version we recall here.
\begin{theorem}[Fractal Takens' Embedding]\label{thm:takens}
    Let $T:U \to U$ be a diffeomorphism of an open set $U\subseteq \mathbb{R}^n$, $A\subseteq U$ be compact and  $m > 2d$ where $d=\textup{boxdim}(A)$. Suppose that the periodic points of $T$ with degree at most $m$ satisfy Assumption~\ref{assumption:1} on $A$. Then, it holds that $x\mapsto (y(x),y(T(x)),\dots, y(T^{m-1}(x)))$ is injective on $A$, for almost all $y\in C^1(U,\mathbb{R}).$ 
\end{theorem}
In Theorem \ref{thm:takens}, $\text{boxdim}(\cdot)$ refers to the box-counting dimension of the set $A$, which reduces to the usual notion of intrinsic dimension whenever $A$ is a manifold; see Definition \ref{def:box}. Moreover, Assumption \ref{assumption:1} is a mild technical assumption on the growth of the number of periodic points of $T$ which is standard in the literature. Lastly, the phrase ``almost all"  comes from the mathematical theory of prevalence~\cite{hunt1992prevalence}, indicating that the set of observables $y\in C^1(U,\mathbb{R})$ for which the conclusion of Theorem \ref{thm:takens} does not hold is negligible in a probabilistic sense; see Appendix \ref{appendix:prev}. For further discussion on Theorem~\ref{thm:takens}, see Appendix \ref{appendix:embed}.

\subsection{Related Work}\label{subsec:related}

In recent years, several works have used invariant measures in data-driven dynamical system modeling. For instance, the Wasserstein distance, derived from optimal transport theory, has been used to compare dynamical attractors~\cite{greve2019data,yang2023optimal}. Moreover, the parameter estimation inverse problem has been reformulated as a PDE-constrained optimization, utilizing the stationary solution of the Fokker--Planck equation as a surrogate for the invariant measure~\cite{yang2023optimal,chen2023detecting}. Pseudo-metrics based on harmonic time averages of observables were introduced for system comparison in~\cite{mezic2004comparison}. Invariant measures, along with other quantities related to the long-term statistics,  have been incorporated as regularization terms when training models according to a trajectory-based mean-squared error loss \cite{jiang2023training,schiff2024dyslim}. Other works have explored optimal perturbations to generate desired linear responses in invariant measures~\cite{galatolo2017controlling}. Furthermore, invariant measures and linear response theory have been applied to study variations in climate systems~\cite{hairer2010simple}. Various works have also explored model architectures and loss functions that reconstruct the invariant measure without directly using the statistical properties of trajectories during training~\cite{linot2023stabilized, park2024when}.
\section{A PDE-Constrained Approach to Dynamical System Identification from Invariant Measures}\label{sec:PDE_IM}

In this section, we introduce a novel PDE-constrained optimization framework for identifying dynamical systems from sparse and noisy observations.  In Section \ref{subsec:FWD}, we describe the forward PDE-based surrogate model $v_{\theta}\mapsto \rho_{\varepsilon}(v_{\theta})$ used for modeling a dynamical system's invariant measure. In Section \ref{sec:gradientcalc}, we then combine the adjoint-state method for PDE-based inverse problems with the backpropagation technique to obtain an efficient gradient calculation for a neural network parameterization of the velocity. Section~\ref{sec:numerics} then showcases the effectiveness of our approach across several numerical experiments. 

\subsection{The Forward Model}\label{subsec:FWD}
\subsubsection{From Trajectories to Densities}\label{subsec:densities}
The Fokker--Planck equation provides a PDE description of the probability density $\rho(x,t)$ of the random variable $X_{t}$ following the SDE \eqref{ES}. In particular, the density evolves as  
\begin{equation}\label{FP2}
    \frac{\partial \rho(x,t)}{\partial t} = -\nabla \cdot (\rho(x,t) v(x)) + D  \Delta \rho(x,t), \qquad x\in \mathbb{R}^d,\qquad t > 0,
\end{equation}
where we have assumed a constant diffusion $\Sigma(x) = D I$, where $I$ denotes the identity matrix and $D \geq 0$ is a constant representing the scale of the diffusion. We remark that if $D = 0$, \eqref{FP2} reduces to the so-called continuity equation, which instead models the probability flow of the ODE given by $\dot{x} = v(x).$ Under certain conditions \cite{huang2015steady}, the steady-state solution $\rho(x)$ of~\eqref{FP2} exists and satisfies
\begin{equation}\label{eq:FPE}
    \nabla \cdot (\rho(x) v(x)) = D \Delta \rho(x),\qquad x\in \mathbb{R}^d.
\end{equation}
Since~\eqref{eq:FPE} describes a limiting distribution $\lim_{t\to \infty} \rho( x,t),$ it has been previously used to approximate invariant measures for stochastically-forced dynamical systems \cite{Allawala_2016}. In what follows, we present an efficient numerical procedure based upon the FVM for computing solutions to \eqref{eq:FPE}. 
\subsubsection{Upwind Finite-Volume Discretization}\label{sec:FVM}

Without loss of generality, we assume that our system evolves on the $d$-dimensional rectangular state space 
$\Omega= [a_1,b_1]  \times \dots \times [a_d,b_d]  \subset \mathbb{R}^d$
with a spatially dependent velocity $v:\Omega \to \mathbb{R}^d$. We define $n_i \in \mathbb{Z}^+$, $1\leq i \leq d$, to be the number of equally-spaced points along the $i$-th spatial dimension at which we wish to approximate the solution of \eqref{FP2}, as well as the mesh spacing $\Delta x_i:= \frac{b_i-a_i}{n_i-1}$. We also discretize the time domain with a step size $\Delta t$. The mesh cells $\{C_j\}_{j=1}^N$, and their corresponding centers $\{x_j\}_{j=1}^N$, are both indexed  using column major order, where $N = \prod_{i=1}^d n_i$. Following \cite{bewley2012efficient}, we implement a first-order upwind finite-volume discretization of the continuity equation, adding a diffusion term using the central difference scheme and enforcing a zero-flux boundary condition \cite{leveque2002finite}. This allows us to obtain an explicit time-evolution of the probability vector $\rho= \begin{bmatrix}\rho_1 & \rho_2 & \dots & \rho_N
\end{bmatrix}^{\top}\in\mathbb{R}^N$.

While $\rho$ is a discrete probability measure over the cells $C_j$, it also corresponds to a piecewise-constant probability density function on $\Omega$. With an abuse of notation, we will refer to both the piecewise-constant density and the discrete probability measure as $\rho$. Based on~\eqref{FP2}, the probability vector at the $l$-th time step evolves as 
$$
\rho^{(l+1)} = \rho^{(l)}+K \rho^{(l)}, \hspace{1cm} K = \sum_{i=1}^d \frac{\Delta t}{\Delta x}K_{i} ,
$$
where each $K_{i}$ is a tridiagonal matrix depending on the upwind velocities
\begin{align*}
v_{j}^{i,-} := \min \big\{0,v_j^i\big\},\quad  v_{j}^{i,+} := \max \big\{0,v_j^i\big\},\quad 
w_{j}^{i,-} := \min \big\{0,w_j^i\big\},\quad  w_{j}^{i,+} := \max \big\{0,w_j^i\big\},
\end{align*}
where  $ v_j^i := v\big(x_{j} - \mathbf{e}_i\Delta x_i/2\big)\cdot \mathbf{e}_i$ and $w_j^i :=v\big(x_{j} + \mathbf{e}_i\Delta x_i/2\big)\cdot \mathbf{e}_i$ denote the $i$-the components of the velocity field at the center of cell faces,  and $\{\mathbf{e}_i\}$ is the standard basis in $\mathbb{R}^d$. For the explicit form of $K$; see Appendix \ref{subsec:FVM_appendix}. To enforce the zero-flux boundary condition, we set both the velocity $v$ and diffusion $D$ to be zero on $\partial \Omega$. Moreover, the Courant--Friedrichs--Lewy~(CFL) stability condition enforces $\Delta t = \mathcal{O}(\Delta x^2)$ to ensure the stability of the scheme. To be more specific, we choose
\[
\Delta t < \frac{1}{2d}\frac{\Delta x^2}{D + \Delta x \|v\|_\infty}\,,
\]
where $\|v\|_\infty = \max_i \| v(x)\cdot \mathbf{e}_i\|_\infty$. In this way, we can enforce that all entries of $I + K$ are non-negative with columns summed to one, which implies that $I + K$ is a Markov matrix. As a result,  the total probability $\rho^{(l)}\cdot \mathbf{1}=1$ is conserved under time evolution, where  $\mathbf{1}:= \begin{bmatrix} 1 & \dots & 1
\end{bmatrix}^\top \in \mathbb{R}^N$. For a complete description of the finite volume scheme, we refer to \cite{leveque2002finite}.

\subsubsection{Teleporation and Diffusion Regularization}\label{subsec:regularization}
We use the finite volume discretization of the Fokker--Planck equation in~\Cref{sec:FVM} to approximate its steady-state solution. After discretization, finding such stationary distributions to~\eqref{eq:FPE} is equivalent to solving the linear system
$
(I+K)\rho =\rho.
$ 
Since the columns of $K$ sum to zero, we have that $M:=I+K$ is a column-stochastic Markov matrix. When $D \neq 0$, $M$ is a transition matrix for an ergodic Markov chain, which has a unique equilibrium. When $D =0$, to guarantee the uniqueness of the equilibrium,~\cite{yang2023optimal} applies the so-called teleportation regularization~\cite{gleich2015pagerank} and considers
$$M_{\varepsilon}:=(1-\varepsilon) M+ \varepsilon\, U,\quad U = N^{-1} \mathbf{1}\mathbf{1}^{\top}  \in \mathbb{R}^{N\times N}.$$
There is now a unique solution to the linear system 
\begin{equation}\label{eq:tele_ls}
M_{\varepsilon} \rho = \rho, \hspace{.5cm} \rho \cdot \mathbf{1}=1, \hspace{.5cm} \rho >0.
\end{equation}

Since $U$ is also a column stochastic Markov matrix with the uniform probability of visiting any point of the mesh, application of $M_{\varepsilon}$ amounts to stopping the dynamics based on $M$ at a random time and restarting it from a uniformly randomly chosen initial point. The size of $\varepsilon$ represents the restarting frequency. The use of teleportation connects all attractors through this ``random restart'', and thus the unique solution $\rho_\varepsilon$ to \eqref{eq:tele_ls} has full support. Similarly, when $D \neq 0$, the Brownian motion connects all disjoint attractors of the deterministic dynamics, giving a unique steady-state solution.

In Figure~\ref{F3}, we illustrate the $\rho_\varepsilon$ computed as the steady-state solution to the Fokker--Planck equation in the top row and the approximation to physical invariant measures of the corresponding SDE in the bottom row. From Figure \ref{F3}, we see that on a  coarse mesh, the first-order finite volume scheme incurs significant numerical error, which gives a computed solution with an artificial diffusion effect and thus is often referred to as the numerical diffusion \cite{bewley2012efficient}. The amount of numerical diffusion is expected to decay as $\mathcal{O}(\max_i{\Delta x_i})$ in the $L^\infty$ norm as we refine the mesh \cite{leveque2002finite}.

\begin{figure}[h!]
\centering
\subfloat[The computed steady state solution to \eqref{eq:FPE} for decreasing values of $\Delta x$.]{
\includegraphics[width = .85\textwidth]{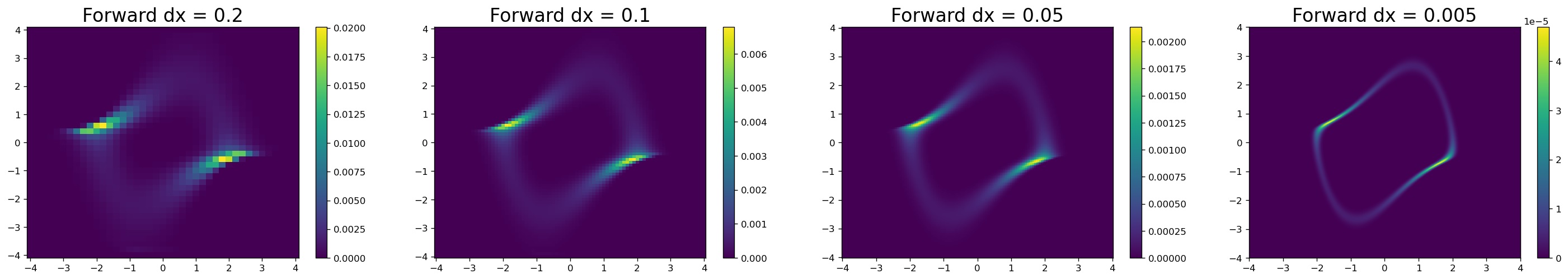}}\\
\subfloat[The occupation measure (see \eqref{eq:occupation}) from the SDE~\eqref{ES} for decreasing values of $\Delta x$.]{\includegraphics[width = .85\textwidth]{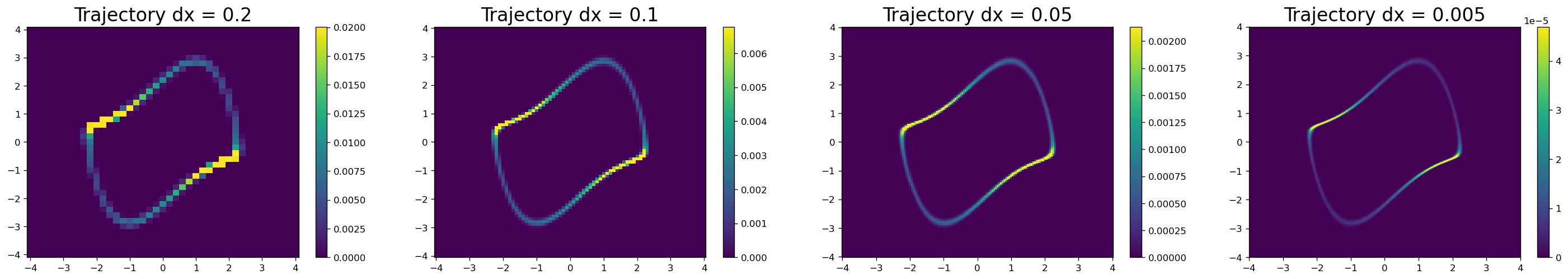}}
\caption{ As the mesh size of the forward model discretization is refined, we visually observe the convergence of the computed steady-state solution (a) to the approximate physical measure (b). The Van der Pol oscillator with $c = 1$ and $D = 0.001$ is used in this example. \label{F3}}
\end{figure}
\subsection{Gradient Calculation}\label{sec:gradientcalc}
The forward model yields a discrete measure $[\rho_1(\theta) \dots \rho_j(\theta) \dots \rho_N(\theta)]^\top$ over the cells $\{C_j\}$. Denoting by $\rho(v_{\theta}) = \rho(\theta)$ the correspoinding piecewise constant density over $\Omega$ (see \cite[Eqn.~(5.1)]{yang2023optimal}) our goal is to solve the optimization problem \eqref{eq:opt2}
by using gradient-based methods. The \textit{adjoint-state method} is an efficient technique by which we can evaluate the derivative $\partial_{\theta} \mathcal{J}$, as the computation time is largely independent of the size of $\theta$. One can derive the adjoint-state method for gradient computations by differentiating the discrete constraint, which in our case is the eigenvector problem:
$$
g(\rho(\theta),\theta) = M_{\varepsilon}(\theta) \rho(\theta) - \rho(\theta)=\mathbf{0},
$$ 
where $\rho(\theta) \cdot \mathbf{1} = 1.$ Specifically, we compute $\partial_\theta \mathcal{J} = \lambda^\top  \partial_\theta g$ where $ \lambda$ solves $\left( \partial_{\rho} g\right)^\top \lambda =  - \left( \partial_{\rho} \mathcal{J}\right)^\top.$
In our case, this linear system is the adjoint equation (see \cite[Eqn.~(5.8)]{yang2023optimal})
\begin{equation}\label{eq:adjoint}
   (M_{\varepsilon}^\top - I)\lambda = - \left( \partial_{\rho} \mathcal{J}\right)^\top +\left( \partial_{\rho} \mathcal{J}\right)^\top \rho \, \mathbf{1}, 
\end{equation}
and the derivative
\begin{equation} \label{eq:adjgradient}
    \partial_\theta \mathcal{J} = \lambda^\top \big(\partial_{\theta} M_{\varepsilon}\big) \rho.
\end{equation}
As a result, we only need to compute the derivatives $\partial_{\rho}\mathcal{J}$ and $\partial_{\theta} M_{\varepsilon}$ to determine the gradient $\nabla_\theta \mathcal{J} = \left(\partial_\theta \mathcal{J} \right)^\top$. The former depends on the choice of the objective function, while the latter is determined by a specific parameterization of the velocity field $v_{\theta}(x)$ determined by its hypothesis space.

Next, we show how to obtain $\partial_\theta M_\varepsilon$, which is the other necessary component in the adjoint-state method for gradient calculation; see~\eqref{eq:adjoint}-\eqref{eq:adjgradient}. To begin with, we consider $\theta =\{ v_j^i\}$ for all $i=1,\dots,d$ and $j = 1,\dots,N$, which corresponds to a piecewise-constant velocity parameterization. Expanding the matrix-vector products in \eqref{eq:adjoint} and taking advantage of the fact that $\partial_{v_j^i}K_i$ is only non-zero in four entries (see Appendix \ref{subsec:FVM_appendix}), we obtain
\begin{align}\label{eq:simplifygrad}
\frac{\partial \mathcal{J}}{\partial v_{j}^i} &=  \lambda \cdot \frac{\partial M_{\varepsilon}}{\partial v_j^i}\rho =   (1-\varepsilon)\frac{\Delta t}{\Delta x}\left( \lambda \cdot \frac{\partial K_i}{\partial v_j^i}\rho \right) \nonumber \\ &=(1-\varepsilon)\frac{\Delta t}{\Delta x}\left( H(v^{i}_j)\rho_{j-S_i} \lambda_j+ (1-H(v^{i}_j)) \rho_{j}\lambda_j -H(v_j^{i}) \rho_{j-S_i}\lambda_{j-S_i} - (1-H(v_j^{i}))\rho_j \lambda_{j-S_i} \right)\nonumber \\
 &= (1-\varepsilon)\frac{\Delta t}{\Delta x} \left(\lambda_j - \lambda_{j-S_i}\right)\left(H(v_j^{i})\rho_{j-S_i} +(1-H(v_j^{i})) \rho_{j}\right),
\end{align}
where $H(\boldsymbol{\cdot})$ is the Heaviside function. Equation~\eqref{eq:simplifygrad} provides an efficient way for computing the gradient of the objective function with respect to the piecewise-constant velocity on cells  $\{C_j\}$  from our finite-volume discretization. 

Alternatively, if the velocity $v = v(x;\theta)$ is smoothly parameterized by $\theta = [\theta_1,\dots,\theta_k,\dots,\theta_p]^\top\in\mathbb{R}^p$, for each $\theta_k$, we can then evaluate
\begin{align}\label{eq:gradient}
    \frac{\partial \mathcal{J}}{\partial \theta_k} &= \sum_{j=1}^N\sum_{i=1}^d \frac{\partial \mathcal{J}}{\partial v_j^i}\frac{\partial v_j^i}{\partial \theta_k},\qquad   \frac{\partial v_j^i}{\partial \theta_k} = \mathbf{e}_i\cdot \frac{\partial v}{\partial \theta_k}\bigg\rvert_{(x_{j} - \mathbf{e}_i\Delta x_i/2;\theta)},
\end{align}
to determine the derivative $\partial_{\theta}\mathcal{J}.$ Motivated by the universal approximation theory of neural networks, one choice is to model the velocity $ v_{\theta}(x)$ as a neural network, where the tunable parameters $\theta\in \mathbb{R}^p$ make up the network's weights and biases. We combine the adjoint-state method for the PDE constraints and the backpropagation technique to update the weights and biases of the neural network. The term $\partial_{v_j^i} \mathcal{J}$ in the gradient calculation~\eqref{eq:gradient} can be computed by first evaluating the neural network on the mesh of cell face centers oriented in the direction of $\mathbf{e}_i$ to obtain $\{v_j^i\}$, which is then plugged into~\eqref{eq:simplifygrad} to obtain $\partial_{v_j^i} \mathcal{J}$. The remaining term $\partial_{\theta} v$ in \eqref{eq:gradient} is then computed via backpropagation.


\subsection{Numerical Results}
\label{sec:numerics}
\subsubsection{Synthetic Test Systems}\label{subsec:vanderpol}
We begin our numerical experiments by considering the autonomous Van der Pol oscillator \cite{1084738}, given by
\begin{equation}\label{eq:VDP}
\begin{cases}
    &\dot{x} = y, \\
    & \dot{y} = c(1-x^2)y - x,
\end{cases}
\end{equation}
where we set $c = 2.$ We compare our approach with the SINDy and Neural ODE frameworks for reconstructing the dynamics \eqref{eq:VDP} based upon a slowly sampled trajectory of length~$10^3$. Both our approach and the Neural ODE are trained using the Adam optimizer \cite{kingma2014adam} and have the same architecture.  As shown in Figure~\ref{fig:initcompare}, all three frameworks perform well when using the quickly sampled trajectory. However, SINDy and the Neural ODE frameworks are less robust to changes in the sampling frequency of the inference data than our approach. This is further demonstrated in Table~\ref{table:compare}, where we quantify the error in the simulated invariant measure using the reconstructed velocity. When the data is sampled at a sufficiently high frequency, Table~\ref{table:compare} also shows that methods such as SINDy or the Neural ODE are preferable in terms of both computational cost and accuracy.
\begin{center}
\begin{table}[h!]
\centering

\begin{tabular}{||c| c |c| c||} 
 \hline
 Method &  Sampling Freq. (Hz) & Wall-Clock Time (s) &Error  \\ [0.5ex] 
 \hline\hline
 SINDy & 10.00 & $2 \cdot 10^{-2}$ & $5.6\cdot 10^{-3}$\\
 Neural ODE & $10.00$ & $5 \cdot 10^2 $ & $5.32\cdot 10^{-3} $ \\
 Ours & $10.00$ & $5\cdot 10^2$ & $1.14\cdot 10^{-1} $ \\
 \hline\hline
  SINDy & 0.25 & $10^{-2}$ & $3.52$\\
 Neural ODE & $0.25$ & $5 \cdot 10^2 $ & $1.81$ \\
 Ours & $0.25$ & $5 \cdot 10^2$  &$6.79 \cdot 10^{-2} $ \\
 \hline
\end{tabular}
\caption{Comparison with the SINDy and Neural ODE frameworks for performing system identification from slowly sampled trajectories. The error is quantified by the 2-Wasserstein distance between the simulated and observed occupation measures (see Appendix \ref{appendix:obj}). \label{table:compare}}

\end{table}
\end{center}

We next demonstrate the applicability of our proposed approach to chaotic dynamical systems. Specifically, we study the Lorenz-63 system \cite{luzzatto2005lorenz}, defined by
\begin{equation}\label{eq:lorenz}
\begin{cases}
    \dot{x} &= c_1(y-x) \\
    \dot{y} &= x(c_2-z)-y\\
    \dot{z} &= xy-c_3z
\end{cases},
\end{equation}
where we consider $(c_1,c_2,c_3) = (10,28,8/3).$ 
For these choices of parameters, the Lorenz-63 system exhibits chaotic behavior and admits a unique physical measure \cite{TUCKER19991197}.
\begin{figure}[h!]
\centering
\subfloat[The reconstructed velocity and simulated trajectories. ]{
\includegraphics[width = .75\textwidth]{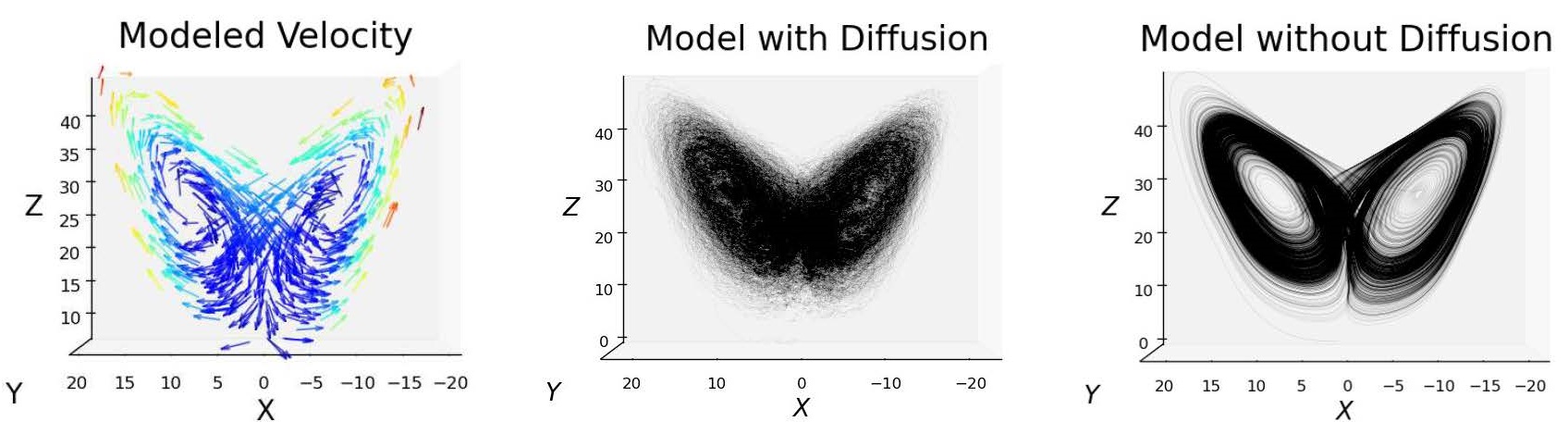}}\\
\vspace{.1cm}
\subfloat[The ground truth velocity and ground truth trajectories.]{\includegraphics[width = .75\textwidth]{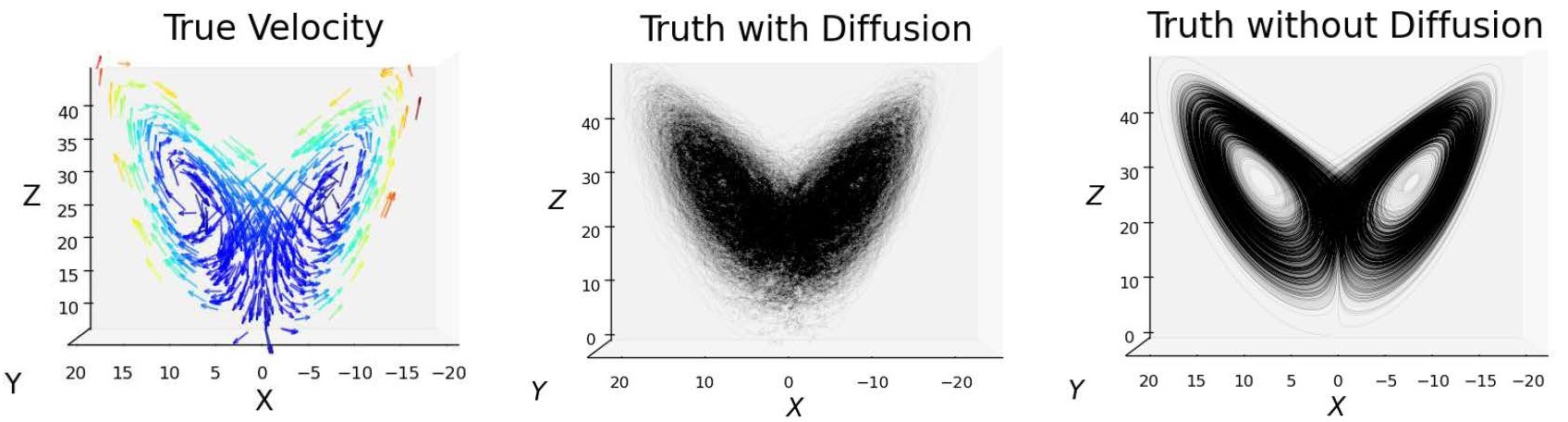}}
\caption{Reconstructing the $\dot{x}$ component of the Lorenz-63 system from its observed, noisy occupation measure. We used a mesh-spacing of $\Delta x = 2$ and a diffusion coefficient of $D = 10.$  \label{fig:lorenz}}
\end{figure}

 In Figure \ref{fig:lorenz}, we assume that the quantities $\dot{y}$ and $\dot{z}$ are known, and we reconstruct the $x$-component of the velocity, using the stochastically-forced Lorenz-63 system's occupation measure. We emphasize that the data used to approximate the Lorenz system's occupation measure can be sampled slowly or even randomly in time; see \cite[Figure 7]{yang2023optimal}. Using the occupation measure \eqref{eq:occupation}, we are able to successfully invert the first component $\dot{x}$ of the Lorenz-63 system's velocity via a neural network parameterization.

\subsubsection{Hall-Effect Thruster}\label{subsec:HET}

We now turn to the more realistic setting of experimentally sampled time-series data, where we also demonstrate the applicability of our approach for performing uncertainty quantification for the forecasted dynamics. Specifically, we study a time-delay embedding (see Section \ref{sec:time_delay}) of the Cathode--Pearson signal sampled from a Hall-effect thruster (HET) in its breathing mode \cite{eckhardt2019spatiotemporal, greve2019data}.  Hall-effect thrusters are in-space propulsion devices that exhibit dynamics resembling stable limit cycles while in breathing mode. Intrinsic physical fluctuations present in the Cathode--Pearson signal indicate that the HET's dynamics may be modeled well by a Fokker--Planck equation. 

\begin{figure*}[h!]
\centering
\subfloat[Inference data]{
\includegraphics[width = .28\textwidth]{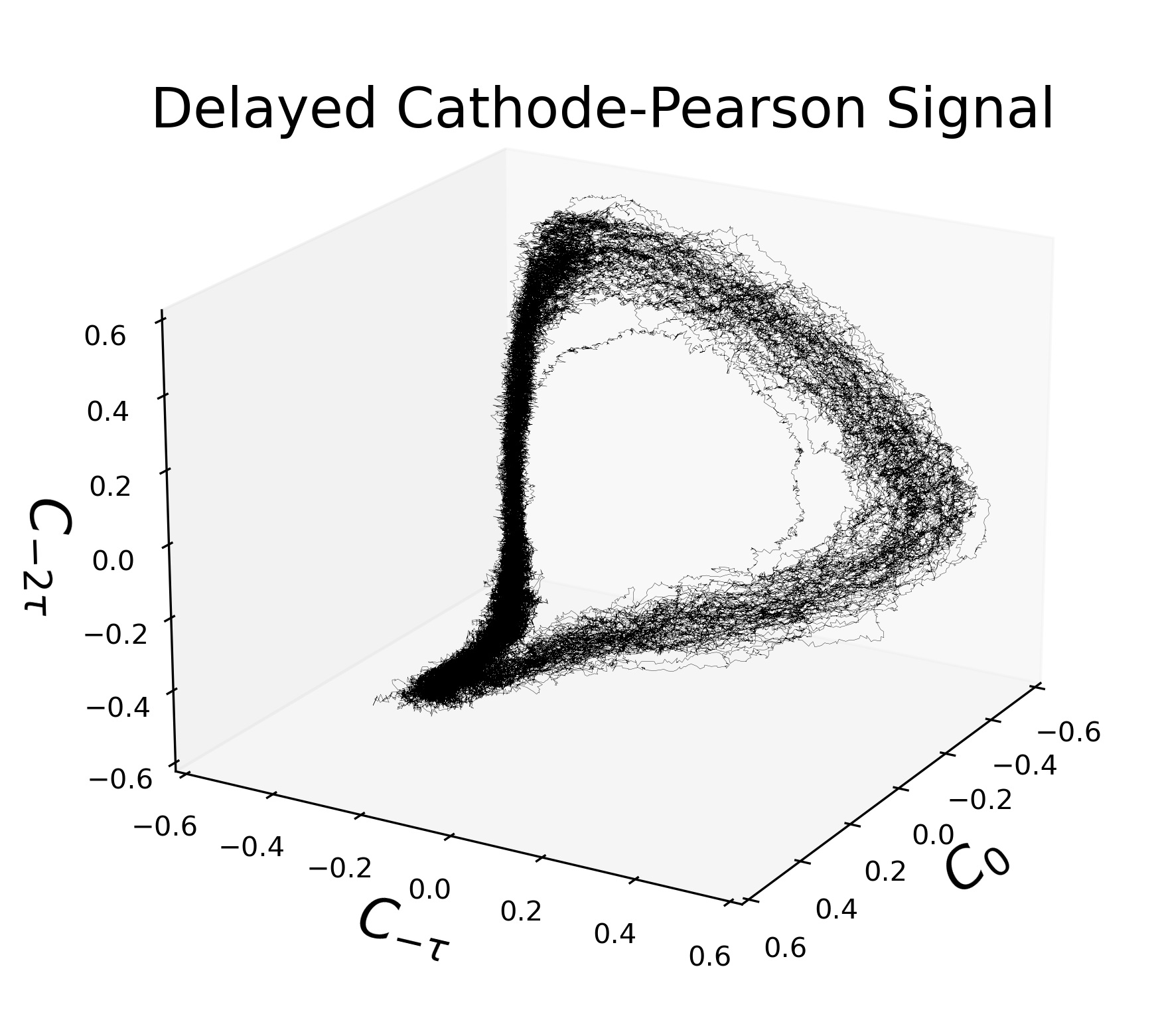}}\label{sf:1}
\subfloat[Reconstructed velocity]{\includegraphics[width = .28\textwidth]{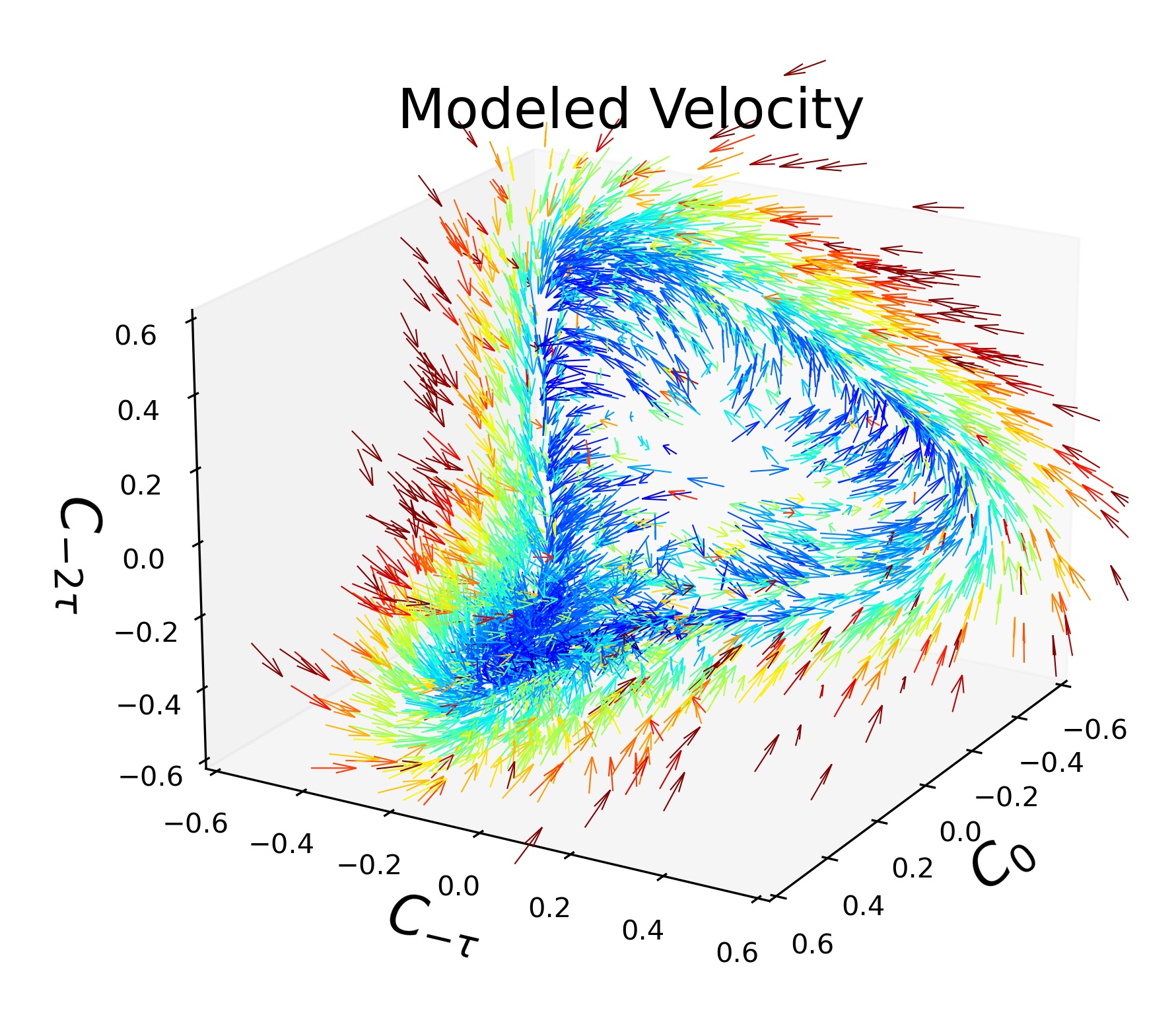}}
\subfloat[Simulated trajectory]{\includegraphics[width = .28\textwidth]{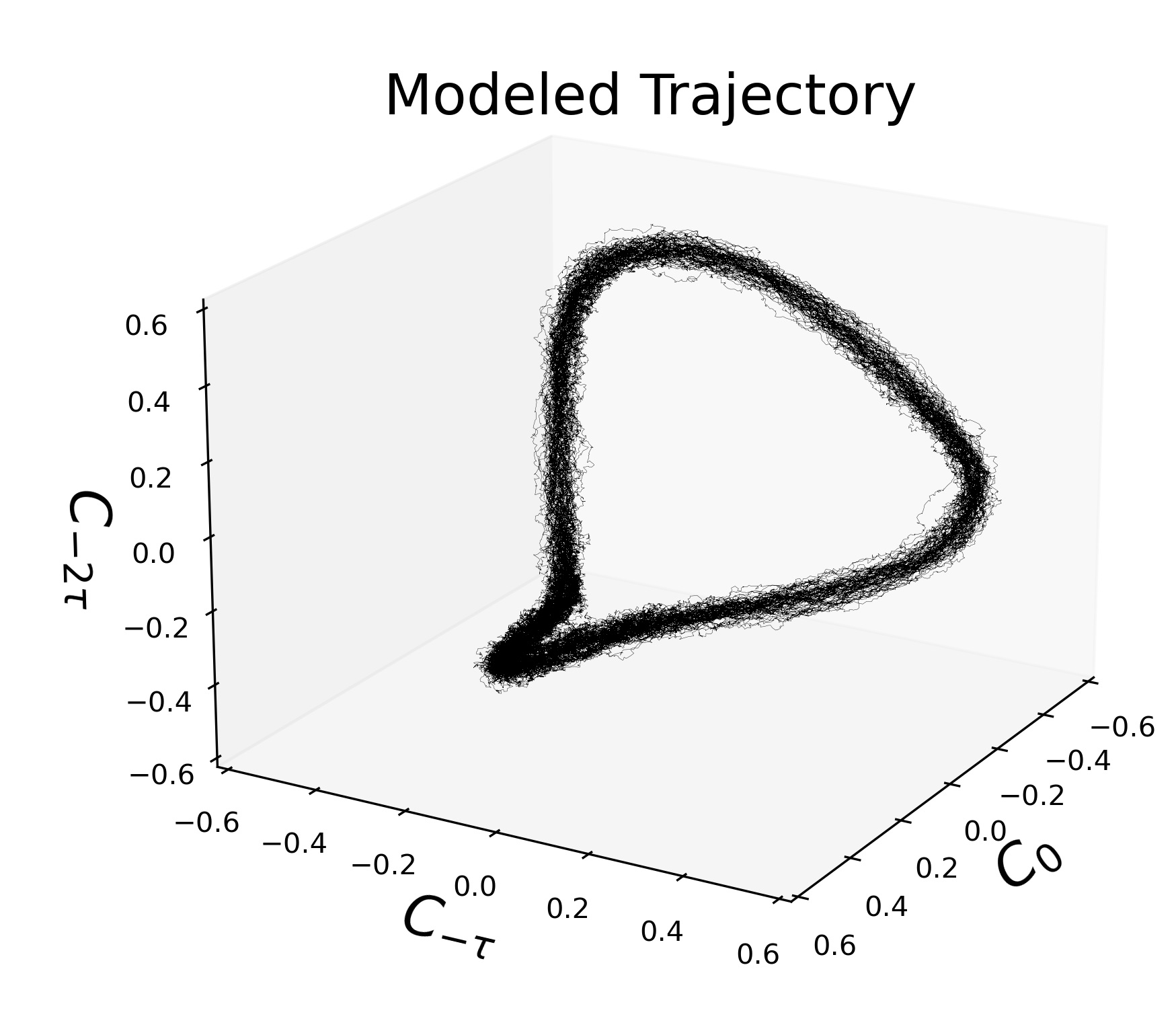} \label{fig:HETc}}\\\vspace{.6cm}
\subfloat[ Using the Fokker--Planck model to predict the evolution of a distribution of initial conditions.]{\includegraphics[width = .9\textwidth]{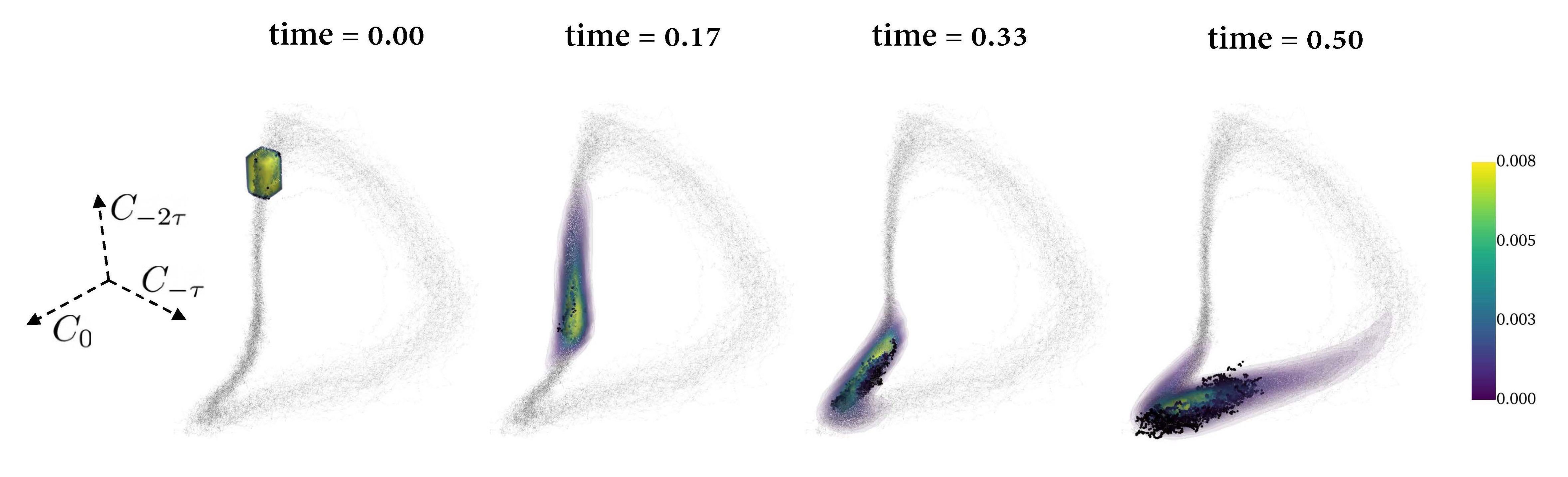} \label{fig:HETd}}\\
\caption{Reconstructing the velocity from the embedded Cathode-Pearson signal's invariant measure. In Figure~\ref{fig:HETc} we simulate a long SDE trajectory according from this model, and in Figure~\ref{fig:HETd} we evolve the reconstructed Fokker--Planck equation forward in time to predict the evolving empirical distribution of certain data samples. \label{fig:HET}}
\end{figure*}
Towards this, we first time-delay embed the observed Cathode--Pearson signal in $d$-dimensions and use a histogram approximation to compute the occupation measure $\rho^*$; see \eqref{eq:opt2}. We then seek a solution to the optimization problem \eqref{eq:opt2} for a velocity $v = v_{\theta}$, which can provide us with a model of the asymptotic statistics of the embedded trajectory. To further ensure that the reconstructed velocity $v_{\theta}$ evolves trajectories according to the same time-scale as the observed trajectory, we rescale both $v_{\theta}$ and the diffusion coefficient $D$ after training. Our results for modeling the HET dynamics are shown in Figure \ref{fig:HET} for an embedding dimension of $d = 3$ and time-delay of $\tau = 1.4\cdot 10^{-5}$ seconds, or rather $\tau = .23$ when normalizing the time-scale to the HET breathing mode frequency (16.6kHz). In Figure \ref{fig:HETc}, we observe that the simulated trajectory of the reconstructed SDE which accurately matches the shape of the embedded Cathode--Pearson signal. In Figure \ref{fig:HETd}, we demonstrate the ability of our approach for performing uncertainty quantification for the forecasted dynamics.

\section{Approximation of the Perron--Frobenius Operator via a Data-Driven Mesh}\label{sec:mesh}

The FVM introduced in Section \ref{sec:FVM} is one approach for approximating a dynamical system's Perron--Frobenius Operator (PFO) \cite{norton2018numerical}; see Section~\ref{subsec:perron}. However, the numerical experiments in Section~\ref{sec:numerics} based upon the FVM quickly become computationally infeasible as the dimension of the state grows; see Table \ref{tab:my_label}.
\begin{table}[h!]
    \centering
    \begin{tabular}{|c|c|c|}
    \hline
    Number of cells     & Wall-clock time (s) & Memory usage (MB)  \\
    \hline
    $\approx 10^2$ & 0.07 & 0.07 \\
    $\approx 10^3$ & 0.08 & 0.39 \\
    $\approx 10^4$ & 0.12 & 3.33\\
    $\approx 10^5$ & 34.04 &320.23 \\
    \hline
    \end{tabular}
    \caption{Cost of the FVM forward evaluation as the number of cells grows.}
    \label{tab:my_label}
\end{table}

In this section, we develop a Galerkin-inspired method for approximating the PFO based upon data-adaptive unstructured meshes and Monte--Carlo integration. 
In Section \ref{subsec:galerkin}, we discuss the assumptions under which a Galerkin projection of the PFO is expected to converge and we introduce a critical regularization parameter that can address an issue of vanishing gradients. In Section~\ref{subsec:MC_analysis}, we study sample-based approximations to this Galerkin projection and perform an analysis of the resulting Monte--Carlo integration error to determine the best placement of mesh cells from the principle of  optimal variance reduction. Numerical experiments and discussions follow in Section~\ref{subsec:numerics2}, where we showcase the applicability of the proposed method for modeling high-dimensional dynamical systems from noisy time-series data.
\subsection{Regularized Projection of the Perron--Frobenius Operator}\label{subsec:galerkin}
Ulam's method is a Galerkin projection of the PFO onto a collection of indicator functions partitioning the state space which preserves critical properties of the PFO, such as positivity and mass-conservation \cite{li1976finite}. Due to the use of non-differentiable characteristic functions, the approach suffers from the problem of vanishing gradients which hinders its utility in gradient-based optimization problems. In this section, we introduce a regularization of this Galerkin projection which utilizes smooth approximations to the indicator functions on cells and can eliminate the problem of vanishing gradients. Our scheme retains the properties of positivity and mass-conservation, and in Theorem \ref{thm:convergence} we prove that as the regularization and discretization parameters are refined, we obtain convergence to the underlying PFO in a suitable operator norm. 

We now begin our technical discussion.  Let $X$ be a compact metric space and let $\mu\in\mathscr{P}(X)$ be fixed. As in Definition \ref{def:PFO}, we consider the Perron--Frobenius operator $\mathcal{P}:L_{\mu}^1(X)\to L_{\mu}^1(X)$ based upon the non-singular discrete map $T:X\to X$. In the continuous-time case, $T$ can simply be taken as the time-$\Delta t$ flow map of some vector field. Now consider a sequence of partitions  $\{\{C_{i,n}\}_{i=1}^n\}_{n=1}^{\infty}$ of $X$ and, for $\varepsilon > 0$, a sequence of partitions of unity $\{\{\psi^{(\varepsilon)}_{i,n}\}_{i=1}^n\}_{n=1}^{\infty}$ belonging to $L_{\mu}^1(X)\cap L_{\mu}^{\infty}(X)$ satisfying Assumption \ref{assumption1} and Assumption \ref{assumption:2}, respectively.
\begin{assumption}\label{assumption1}
The sequence $\{\{C_{i,n}\}_{i=1}^n\}_{n=1}^{\infty}$ satisfies the following. 
    \begin{itemize}
\item \textit{(Partition)}: $C_{i,n}\cap C_{j,n} = \emptyset$ if $i \neq j$ and $\bigcup_{i=1}^n C_{i,n} = X.$
    \item \textit{(Increasing resolution)}: $\displaystyle\lim_{n\to\infty}\max_{1\leq i\leq n}(\textup{diam}(C_{i,n})) = 0$.
    \item \textit{(Positive measure)}: $\mu(C_{i,n}) > 0,$ for all $1\leq i \leq n$ and all $n\in\mathbb{N}$.
\end{itemize}
\end{assumption}
\begin{assumption}\label{assumption:2}
    The collection $\{\{\psi^{(\varepsilon)}_{i,n}\}_{i=1}^n\}_{n=1}^{\infty}$ satisfies the following properties.
    \begin{itemize}
        \item (\textit{Partition of unity}): For $\varepsilon > 0$ and $n\in\mathbb{N}$ fixed,
        $\sum_{i=1}^n \psi_{i,n}^{(\varepsilon)}(x) = 1$ and $\psi_{i,n}^{(\varepsilon)}(x)\geq 0.$ 
        \item (\textit{Approximation}): For $n\in\mathbb{N}$ and $i\leq n$ it holds $\psi_{i,n}^{(\varepsilon)} \xrightarrow[]{\varepsilon \to 0}\chi_{C_{i,n}}$ pointwise $\mu$-a.e.
    \end{itemize}
\end{assumption}
Note that $n\in\mathbb{N}$ represents the resolution of a given partition $\{C_{i,n}\}_{i=1}^n$, which for notational convenience we will occasionally suppress. Now, we define 
\begin{equation}\label{eq:phi}
    \phi_{i}:= \frac{\chi_{C_{i}}}{\mu(C_{i})}\in L_{\mu}^1(X) \cap L_{\mu}^{\infty}(X),\qquad 1\leq i \leq n,
\end{equation}
which is a normalized characteristic function on the cell $C_{i}.$ It also holds that $\phi_{i} d\mu = d\mu_{i},$ where $\mu_{i}\in\mathscr{P}(X)$ is the conditional measure defined by $\mu_i(B) :=\mu(C_i\cap B)/\mu(C_i)$, for all $B\in\mathscr{B}$. The subspace of $L_{\mu}^1(X)$ spanned by the collection \eqref{eq:phi} will be denoted  $\Delta_n:=\text{span}(\{\phi_{i,n}\}_{i=1}^n) \subseteq L_{\mu}^1(X)$. It is helpful to define the operator $\mathcal{Q}^{(n)}:L_{\mu}^1(X)\to \Delta_n$, where  
\begin{equation}\label{eq:proj}
    \mathcal{Q}^{(n)} f := \sum_{i=1}^n a_{i,n} \phi_{i,n} ,\qquad a_{i,n} := \int_{C_{i,n}}fd\mu.
\end{equation}
Intuitively, $\mathcal{Q}^{(n)}f$ is a piecewise constant function describing the average value of $f$ over each cell.
Now, for $n\in\mathbb{N}$ and $\varepsilon > 0$, we define $\mathcal{L}^{(n,\varepsilon)}:\Delta_n\to \Delta_n$ by setting 
\begin{equation}\label{eq:ulam_projection_eps}
    \mathcal{L}^{(n,\varepsilon)}\phi_{i,n} := \sum_{j=1}^n M_{i,j}^{(n,\varepsilon)} \phi_{j,n}, \qquad M_{i,j}^{(n,\varepsilon)}:= \int_{X}(\psi_{j,n}^{(\varepsilon)}\circ T  )\phi_{i,n} d\mu. 
\end{equation}
In the case when $\psi_{j,n}^{(\varepsilon)} = \chi_{j,n}$, we have that 
\begin{equation}\label{eq:cond_int}
    M_{i,j} = \int_X \chi_{T^{-1}(C_j)} \phi_{i}d\mu =  \int_X \chi_{T^{-1}(C_j)}d\mu_i,
\end{equation}
and thus \eqref{eq:ulam_projection_eps} reduces to the usual Ulam method in which the entry $M_{i,j}$ describes the transition probability between cell $C_{i}$ and cell $C_{j}$, under the dynamics $T$; see \cite{li1976finite}. Using the operator $\mathcal{L}^{(n,\varepsilon)}:\Delta_n\to \Delta_n$, we finally define $\mathcal{P}^{(n,\varepsilon)}:L_{\mu}^1(X)\to L_{\mu}^1(X)$ by setting $\mathcal{P}^{(n,\varepsilon)}:= \mathcal{L}^{(n,\varepsilon)}\mathcal{Q}^{(n)},$ which we regard as our approximation to the true PFO. Theorem \ref{thm:convergence} establishes the main approximation properties of $\mathcal{P}^{(n,\varepsilon)}$. In particular, for any finite $n\in\mathbb{N}$ and $\varepsilon > 0$ we show that the operator preserves positivity and satisfies the mass-conservation principle. Moreover, as the parameters $n\in\mathbb{N}$ and $\varepsilon > 0$ are refined, we obtain convergence to the true PFO in the $\|\cdot \|_{L_{\mu}^1\to L_{\mu}^1}$ norm.
\begin{theorem}\label{thm:convergence}
    For $\varepsilon > 0$ and $n\in\mathbb{N}$ fixed,  the operator $\mathcal{P}^{(n,\varepsilon)}:L_{\mu}^1(X)\to L_{\mu}^1(X)$ is Markov. Moreover, it holds that  $\displaystyle \lim_{n\to\infty}\lim_{\varepsilon \to 0}\|\mathcal{P}^{(n,\varepsilon)} - \mathcal{P}\|_{L_{\mu}^1\to L_{\mu}^1} = 0.$
\end{theorem}
Our proof of Theorem~\ref{thm:convergence} is presented in Appendix \ref{appendix:proj_proof}. 
\subsection{Optimal Partition Construction}\label{subsec:MC_analysis}
While in Section \ref{subsec:galerkin}, we established the convergence of a PFO approximation scheme in the infinite-data and infinite-resolution limit (see Theorem \ref{thm:convergence}), in practice we must always perform computations using a finite amount of data and a finite number of cells. In this section, we investigate how the construction of the partition $\{C_i\}_{i=1}^n$ impacts our approximation error. While the optimal choice of partition for Ulam's method has previously been studied in the literature in the infinite-data limit, see e.g.~\cite{murray2004optimal}, we focus  on characterizing optimality from the perspective of finite-sample variance reduction. We formalize our optimality principle as a minimax problem and we show that  constructing a partition which \textit{evenly distributes} the observed data across the cells satisfies this optimality principle (see Proposition \ref{prop:optimal_strat}). In practice, this can be approximately achieved using a $k$-means clustering of the observed dataset, and in certain situations a constrained $k$-means routine can build clusters which exactly satisfy this condition \cite{bradley2000constrained}.

For simplicity, we consider the case without regularization, i.e., we assume $\psi_j^{(\varepsilon)} = \chi_{C_j}$ for $1\leq j \leq n$. We assume access to $N$ i.i.d.~samples $\{x_k\}_{k=0}^{N-1}\sim \mu\in\mathscr{P}(X)$, as well as the pairings $\{(x_k,T(x_k)\}_{k=1}^N$ under some dynamical system $T:X\to X$. That is, we do not know the evolution rule $T$ a-priori, but we instead observe $N$~i.i.d. samples describing its action on a fixed data-distribution. We will also assume throughout that $\text{mod}(n,N)  = 0$, i.e., it is possible to split the observed samples into $N/n$ distinct subsets. 

Let us now write $N_i:=|\{x_k\in C_i\}|$ to denote the number of samples in the $i$-th cell, which we assume to be positive. We also define $\{x_k^i\}_{k=1}^{N_i}$ as the subset of $\{x_k\}_{k=1}^N$ which is contained in the cell $C_i$.
Since the samples $\{x_k\}_{k=1}^N$ are i.i.d.~with respect to $\mu$, it also holds that the collection $\{x_k^i\}_{k=1}^{N_i}$ of samples which are contained in $C_i$ are i.i.d.~with respect to $\mu_i$; see Section \ref{subsec:galerkin}. In practice, one represents the PFO using the matrix $M\in\mathbb{R}^{n\times n}$ defined in \eqref{eq:cond_int}, and with access to only finitely many samples from $\mu$, the entries of $M$ are approximated via Monte--Carlo integration:
\begin{equation}\label{eq:MC}
    \widehat{M}_{i,j}= \frac{1}{N_i}\sum_{k=1}^{N_i} \chi_{T^{-1}(C_j)}(x_{k}^i),\qquad1\leq i,j\leq n.
\end{equation}
The approximation \eqref{eq:MC} has known standard deviation
\begin{equation}\label{eq:err}
    \mathcal{D}_{i,j}:=\Big(\mathbb{E}[(M_{i,j}-\widehat{M}_{i,j})^2]\Big)^{1/2} = \sqrt{\frac{\text{Var}_{\mu_i}(\chi_{T^{-1}(C_j)})}{N_i}}, \qquad 1\leq i,j\leq n ;
\end{equation}
see \cite[Section 2]{Caflisch_1998}. We now state our principle of optimal variance reduction, which seeks a partition of $X$ minimizing the worst-case standard deviation.
\begin{variance1}
    Let $\mathcal{X}_n$ denote the set of all  partitions of $X$ into $n$ subsets. Then, $\{\hat{C}_{i}\}\in \mathcal{X}_n$ satisfies the principle of optimal variance reduction (V1) if 
    \begin{equation}\label{eq:variance_1}
        \{\hat{C}_i\} \in \argmin_{\{C_i\} \in \mathcal{X}_n } \max_{1\leq i,j \leq n} \mathcal{D}_{i,j},
    \end{equation}
    where $\mathcal{D}_{i,j}$ is the standard deviation defined in \eqref{eq:err} based on a partition $\{C_i\} \in \mathcal{X}_n$
\end{variance1}
Without prior knowledge of the underlying system $T$ or the measure $\mu$, which are both required to evaluate the numerator of \eqref{eq:err}, we cannot check if a given partition $\{C_i\}\in \mathcal{X}_n$ satisfies our principle of optimal variance reduction. Notice that the numerator of \eqref{eq:err} is given by 
$$\text{Var}_{\mu_i}(\chi_{T^{-1}(C_j)}) =  \int_X \chi_{T^{-1}(C_j)} d\mu_i  - \bigg(\int_X \chi_{T^{-1}(C_j)}d\mu_i\bigg)^2 = M_{i,j}-M_{i,j}^2.$$
Since $M_{i,j}-M_{i,j}^2\in [0,1/4]$, we instead consider a modification to our principle of optimal variance reduction which places a worst-case bound on the standard deviation \eqref{eq:err} and can be verified when only finite sample data is available.
\begin{variance2}
A partition $\{\hat{C}_i\}\in \mathcal{X}_n$ satisfies the principle of optimal variance reduction (V2) if 
\begin{equation}\label{eq:solve_arg}
 \{\hat{C}_i\}\in \argmin_{\{C_i\}\in \mathcal{X}_n} \max_{1\leq i \leq n}\frac{1}{\sqrt{N_i}}.   
\end{equation}
\end{variance2}
Note that the minimax problem \eqref{eq:solve_arg} depends only on the number of points contained in each cell. Thus, solving \eqref{eq:solve_arg} is equivalent to finding numbers $N_1,\dots,N_n\in\mathbb{N}$ which solve
\begin{equation}\label{eq:minimizesum}
   \argmin_{N_1,\dots,N_n\in\mathbb{N}} \max_{1\leq i \leq n}\frac{1}{\sqrt{N_i}} \qquad \text{such that}\qquad \sum_{i=1}^n N_i = N.
\end{equation}
By inspection, one can verify that the solution to \eqref{eq:minimizesum} is given by $N_i = N/n$ for each $1\leq i \leq n$. This leads us to Proposition \ref{prop:optimal_strat} which classifies the set of partitions satisfying the principle of optimal variance reduction (V2).
\begin{proposition}\label{prop:optimal_strat}
    A partition $\{\hat{C}_i\}\in \mathcal{X}_n$  solves \eqref{eq:solve_arg} if and only if $N_i = N/n$ for all $1\leq i \leq n$. 
\end{proposition}

Proposition \ref{prop:optimal_strat} indicates that we should build partitions which contain the same number of samples in each cell, in order to reduce the worst-case standard deviation when approximating the PFO via the Monte--Carlo integration \eqref{eq:MC}. Notably, if the data samples come from a long trajectory which samples the invariant measure, then this mesh-construction strategy will involve focusing the majority of cells on the support of the invariant measure. We therefore expect this scheme to be especially useful for approximating the PFO of high-dimensional systems admitting low-dimensional invariant measures. In Section \ref{subsec:numerics2}, we implement the selection principle suggested by Proposition \ref{prop:optimal_strat} via $k$-means clustering and demonstrate the improved efficiency of the data-adaptive unstructured mesh compared to a uniform mesh; see Figure \ref{fig:Cat}.
\subsection{Numerical Results}\label{subsec:numerics2}
\subsubsection{Modified Cat Map}\label{subsubsec:cat}

We begin by investigating the efficiency of the uniform and unstructured mesh approaches for approximating the invariant density of a two-dimensional discrete-time system. In particular, we consider the map $S:[0,1]^2\to [0,1]^2$ defined by
\begin{equation}\label{eq:cat}
    S:= g\circ C \circ g^{-1},\qquad g(x,y):=(x^{10},y), \qquad C(x,y):= (2x+y,x+y)\pmod{1},
\end{equation}
where $C:[0,1]^2\to [0,1]^2$ is the prototypical Arnold cat map \cite{lasota2013chaos}, which admits the Lebesgue measure on $[0,1]^2$ as its unique physical invariant measure. We have composed with $g$ and $g^{-1}$ in \eqref{eq:cat} to obtain a new dynamical system which instead admits $\rho^*(x,y) = 10x^9$ as the unique invariant density.

\begin{figure}[h!]
    \centering
    \subfloat[Visualization of the ground truth invariant density and the density approximation by both a uniform mesh and the data-adaptive unstructured mesh using 400 cells.]{\label{fig:CatA} \includegraphics[width=\textwidth]{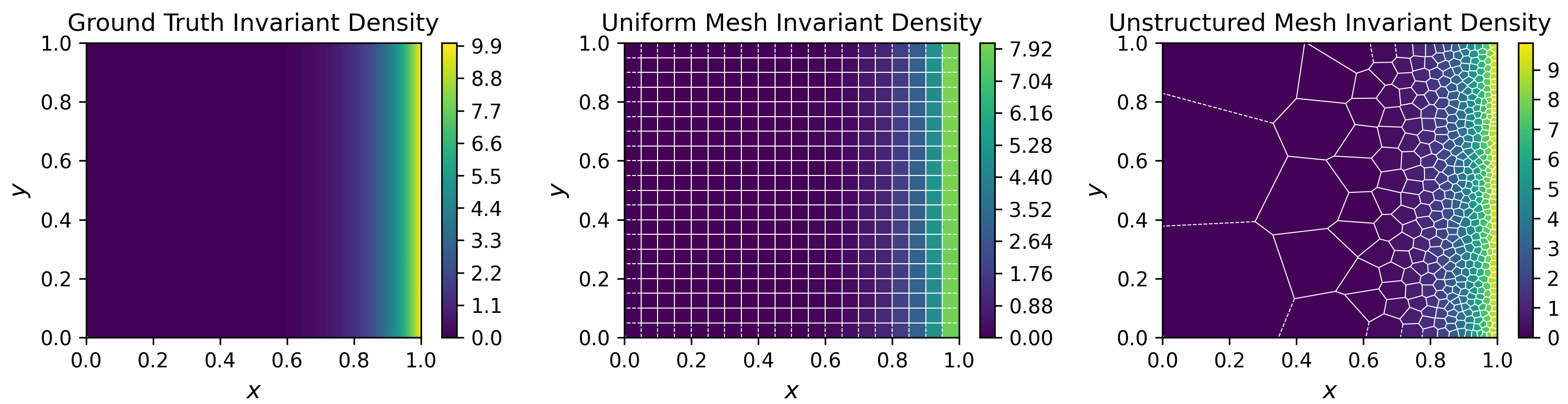}}\\
    \vspace{.5cm}
    \subfloat[Convergence comparison between the uniform and unstructured meshes in the large-data limit.]{\label{fig:CatB}\includegraphics[width = .35\textwidth]{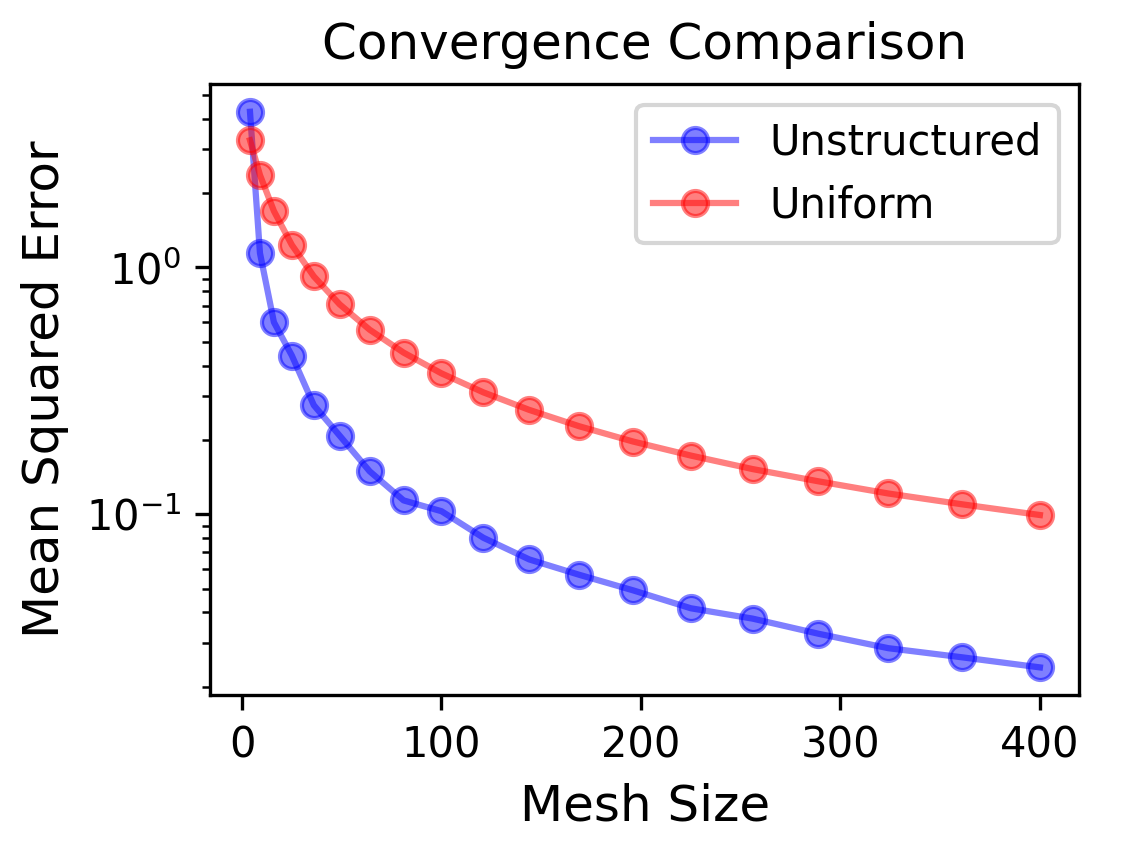}}\qquad \subfloat[Efficiency comparison between the uniform and unstructured mesh approaches for a fixed number of cells.]{\label{fig:CatC}\includegraphics[width = .36\textwidth]{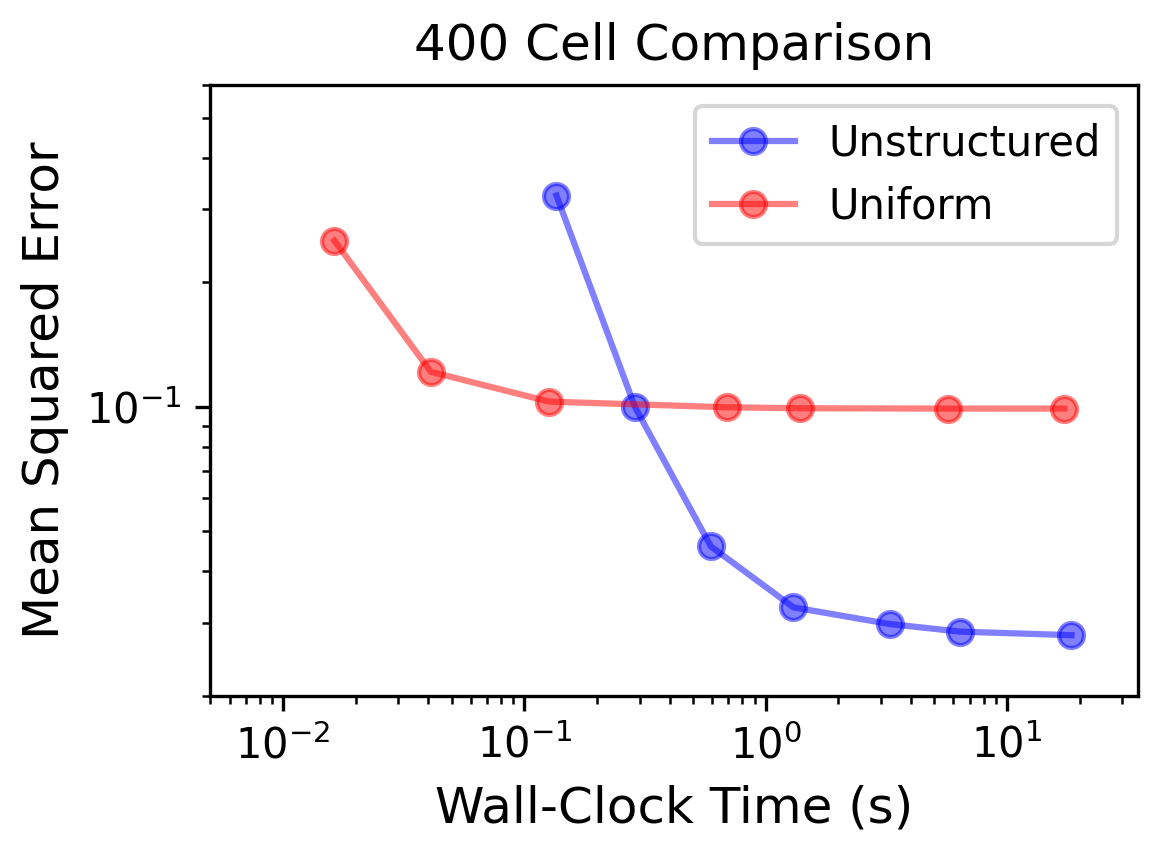}}
    \caption{Comparison between the uniform and unstructured mesh approaches for approximating the invariant density of \eqref{eq:cat}. In Figure \ref{fig:CatA} and Figure \ref{fig:CatB}, we consider a dataset with $N = 10^4$ initial conditions and $K = 10^3$ iterations of \eqref{eq:cat}, and in Figure \ref{fig:CatC}, we vary $K\in 10,10^4]$.  }
    
    \label{fig:Cat}
\end{figure}

In Figure \ref{fig:Cat}, we compare the  uniform and unstructured mesh approaches for approximating the invariant density $\rho^*$ from finite-sample data.  In particular, we consider a dataset of the form $\{\{S^{\ell}(x_k)\}_{k=1}^N\}_{\ell = 1}^{K},$ where $\{x_k\}_{k=1}^N\sim \mathcal{U}([0,1]^2)$ are i.i.d.~uniformly distributed initial conditions. In this example, $N$ represents the number of initial conditions and $K$ is the number of iterations of the map $S$ applied to each initial condition. Using this dataset, we construct the Markov matrix following~\eqref{eq:MC} and extract the approximate invariant measure as its dominant eigenvector. Even for small values of $K$, the data distribution is concentrated near the effective support of the invariant measure. Thus, we observe a fine resolution of the unstructured mesh cells near the region of highest density, whereas the uniform mesh covers all portions of the state-space equally; see Figure \ref{fig:CatA}. For a fixed number of cells, we find that the unstructured mesh achieves lower error than the uniform mesh, and is therefore more memory-efficient; see Figure \ref{fig:CatB}. The main computational cost of the unstructured mesh approach involves determining the cell containment of every sample. In Figure \ref{fig:CatC}, we vary the size of $K$ to study the trade-off in efficiency when using the unstructured mesh approach compared to the uniform mesh with a fixed number of $n = 400$ cells. We find that the unstructured mesh approach converges to a lower error solution with comparable computational cost for a reasonable mesh size.

\subsubsection{30-Dimensional Lorenz-96 System}

We now apply the unstructured mesh approach for approximating the PFO to a high-dimensional Lorenz-96 system \cite{karimi2010extensive}, given by 
\begin{equation}\label{eq:L96}
\dot{x}_i =   (x_{i+1}-x_{i-2})x_{i-1}-x_i + 8,
\end{equation}
where $x_0 = x_d$, $x_{-1} = x_{d-1},$ and $x_{d+1} = x_1$ and $d = 30$ is the state-dimension. Rather than recovering the approximate invariant measure from trajectory samples, as in Section \ref{subsubsec:cat}, we now reconstruct the parameterized vector field $v_{\theta}.$ Note that the FVM-based approach introduced in Section \ref{sec:PDE_IM} is intractable for solving this inverse problem in high-dimensions (see Table~\ref{tab:my_label}), so we instead rely on the unstructured mesh approach. Moreover, the improved computational efficiency of the data-adaptive mesh and Monte--Carlo sampling approach enables us to compare the full Markov matrices during optimization, rather than their dominant eigenvectors as in Section~\ref{sec:PDE_IM}.

\begin{figure}[h!]
    \centering
   \subfloat[Three dimensional projections of a ground truth trajectory (left) and a trajectory from our reconstructed model (right).]{\includegraphics[width=0.8\linewidth]{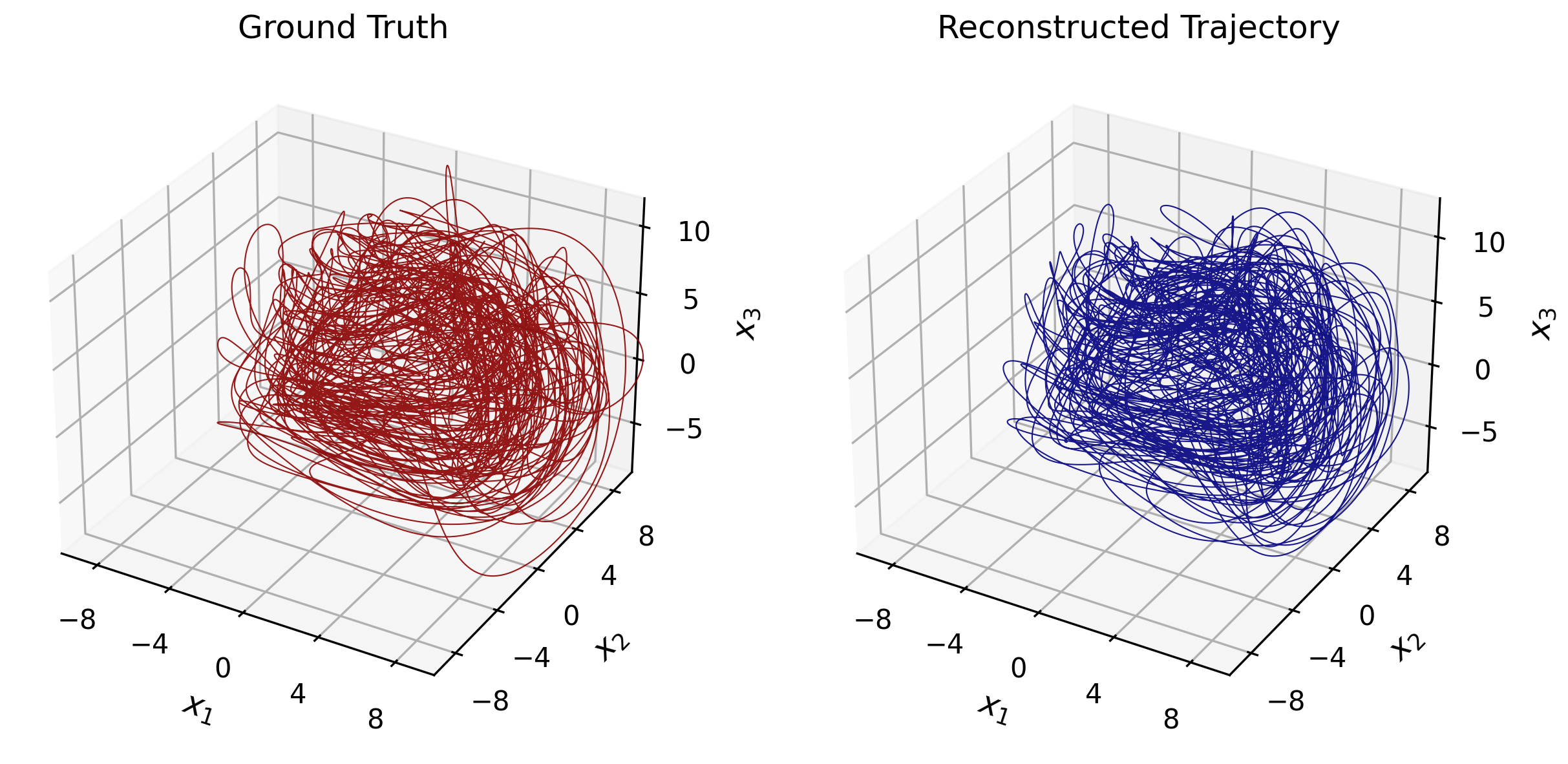} \label{fig:L96a}} \\\vspace{.2cm}
    \subfloat[Visualizations of a short time-trajectory from both the ground truth system (left) and our reconstructed model (right). Each row of the plot corresponds to a fixed state variable, with the color indicating its value at a given time.]{\includegraphics[width=0.8\linewidth]{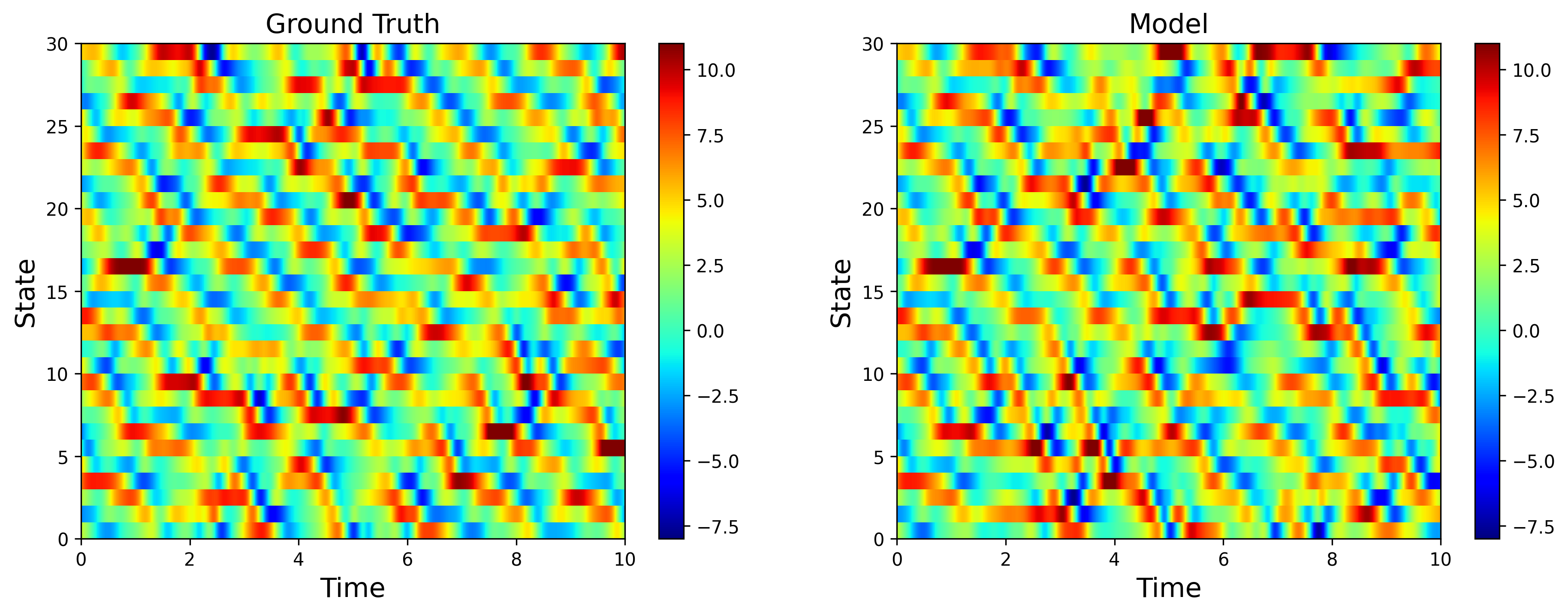} \label{fig:L96b}}
    \caption{Reconstructing the first component of the Lorenz-96 system's velocity $\dot{x}_1 = v_{\theta}(x)$ with a neural network parameterization. Figure \ref{fig:L96a} compares a projection of a long trajectory of the model with the ground truth, while Figure \ref{fig:L96b} compares the time-evolution of all 30 state variables over a short trajectory. In both comparisons, our model obtains close agreement with the ground truth system.}
    \label{fig:L96}
\end{figure}

Thus, if $v_{\theta}$ denotes the parameterized velocity we wish to reconstruct based upon the time-series measurements $\{x^*(t_i)\}_{i=1}^N$, where $t_i = t_0 + i\Delta t,$ we now solve 
\begin{equation}\label{eq:markopt}
    \min_{\theta\in \Theta} \mathcal{J}(\theta),\qquad \mathcal{J}(\theta) = \|M^{(\varepsilon)}(v_{\theta}) - M^*\|,
\end{equation} 
where $M^{(\varepsilon)}(v_{\theta})$ is the regularized Markov matrix constructed based upon the time $\Delta t$ flow map of the vector field $v_{\theta}$ and $M^*$ is the Markov matrix from the observed trajectory samples. We perform a $k$-means clustering of the observed trajectory $\{x^*(t_i)\}_{i=1}^N$ to obtain our unstructured mesh cells with centers $\{c_i\}_{i=1}^n$, which remain fixed during optimization. We then construct the partition of unity $\{\psi_i^{(\varepsilon)}\}_{i=1}^n$, which is used to form the regularized Markov matrix~\eqref{eq:cond_int}, by setting 
$$\psi_i^{(\varepsilon)}(x):= \frac{r_i(x;\varepsilon)}{\sum_{i=1}^nr_i(x;\varepsilon)},\qquad r_i(x;\varepsilon) = \log(1+e^{-|c_i - x|/\varepsilon}),\qquad x\in \mathbb{R}^d,\qquad \varepsilon > 0.$$
We find that choosing a relatively large value of $\varepsilon$ improves the robustness of the reconstruction~\eqref{eq:markopt} to noise in the observed data.

In Figure \ref{fig:L96}, we invert the first component of the 30-dimensional Lorenz-96 system's velocity via a neural network parameterization. We use the Adam optimizer \cite{kingma2014adam} to reduce the Markov matrix objective \eqref{eq:markopt} with $\varepsilon = 5$ and $ n = 200$ cells. That is, we parameterize $\dot{x}_1 = v_{\theta},$ assuming all other velocity components are known and seek to learn the map from $\mathbb{R}^{30}\to \mathbb{R}.$ To evaluate the success of our reconstructed model, we visualize a long trajectory of the 30-dimensional system in a projected coordinate frame (see Figure \ref{fig:L96a}) as well as a short time-series including all 30 state variables (see Figure~\ref{fig:L96b}). The results of Figure~\ref{fig:L96} show close agreement between our model and the ground truth, indicating that our proposed framework for approximating the PFO is well-equipped for modeling high-dimensional dynamical systems. 

\section{Invariant Measures in Time-Delay Coordinates for Unique Dynamical System Identification}\label{sec:delay_IM}

While Section \ref{sec:mesh} studied a computational extension of the PDE-constrained optimization approach to velocity reconstruction from invariant measures presented in Section \ref{sec:PDE_IM}, we now improve the well-posedness of the inverse problem. The inverse problem \eqref{eq:opt2} which is 
solved in Section \ref{sec:PDE_IM} turns out to be highly ill-posed, due to a lack of uniqueness; see Figure \ref{fig:torus}. In this section, we establish that under a simple change of coordinates, the inverse problem solution is unique. Along this direction, we introduce a theoretical framework that rigorously justifies the use of invariant measures in time-delay coordinates for reconstructing dynamical systems from data; see Figure \ref{fig:flowchart} for a flowchart highlighting our main contributions in this section. In Section \ref{subsec:invariant_delay}, we study how invariant measures are affected by a change of coordinates to precisely define what is meant by an invariant measure in time-delay coordinates, and we state our main theoretical results of the section. Discussions and numerical experiments follow in Section \ref{subsec:discussions2} and Section \ref{subsec:numerics3}, respectively.


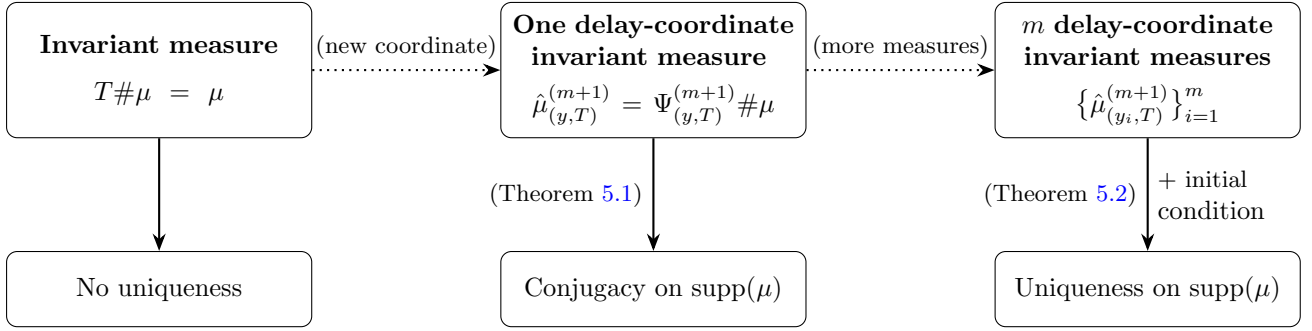
\begin{figure*}[htb!]
\begin{tikzpicture}[scale = 0.7,
  box/.style={draw, minimum width=3.8cm, minimum height=1.8cm, text centered, rounded corners, text width=3.8cm, align=center},
  subbox/.style={draw, minimum width=3.2cm, minimum height=1cm, text centered, rounded corners, text width=3.8cm, align=center},
  arrow/.style={-Stealth, thick},
dottedarrow/.style={-Stealth, thick, dotted},
  every node/.append style={font=\small},
  node distance=1.5cm
]

\node[box] (box1) {\textbf{Invariant measure}\vspace{.2cm}\\
$T\# \mu = \mu$};
\node[box, right=2.5cm of box1] (box2) {
 \textbf{One delay-coordinate invariant measure} \vspace{.2cm} \\ $\hat{\mu}_{(y,T)}^{(m+1)} = \Psi_{(y,T)}^{(m+1)}\# \mu$};
\node[box, right=2.5cm of box2] (box3) {
\textbf{$m$ delay-coordinate invariant measures}\\ \vspace{.2cm} $\big\{ \hat{\mu}_{(y_i,T)}^{(m+1)}\big\}_{i=1}^{m}$};

\node[subbox, below=of box1] (subbox1) {No uniqueness};
\node[subbox, below=of box2] (subbox2) {
Conjugacy on $\text{supp}(\mu)$};
\node[subbox, below=of box3] (subbox3) {
Uniqueness on $\text{supp}(\mu)$};

\draw[dottedarrow] (box1) -- node[above] {\footnotesize(new coordinate)} (box2);
\draw[dottedarrow] (box2) -- node[above] {\footnotesize(more measures)} (box3);
\draw[arrow] (box1) -- node[right,align = left] {} (subbox1);
\draw[arrow] (box2) -- node[left,align = left] {\footnotesize
 (Theorem \ref{thm:1})} (subbox2);
\draw[arrow] (box3) -- node[left,align = left] {\footnotesize
 (Theorem \ref{thm:2})} (subbox3);
\node[above=0.75cm of subbox3, right, align = left]{\footnotesize
 + initial\\
 condition};
\end{tikzpicture}
\caption{Flowchart of our main results. While  the invariant measure $\mu$ cannot uniquely identify the dynamical system, one delay-coordinate invariant measure can ensure topological conjugacy of the reconstructed system on $\text{supp}(\mu)$. Moreover, several delay-coordinate invariant measures can  uniquely determine the dynamics on $\text{supp}(\mu)$, with an appropriate initial condition. }
\label{fig:flowchart}
\end{figure*}
\subsection{Invariant Measures in Time-Delay Coordinates}\label{subsec:invariant_delay}
We will next study how invariant measures are transformed under a coordinate change of the dynamics. The dynamical systems $T:X\to X$ and $S:Y\to Y$ are said to be \textit{topologically conjugate} (related by a change of coordinates) if there exists a homeomorphism $h:X\to Y$, such that $S = h\circ T \circ h^{-1}$. The conjugating map $h$ should be viewed as a nonlinear coordinate change between the spaces $X$ and $Y$ which maps between trajectories of the systems $T$ and $S$. The following proposition relates the invariant measures of conjugate dynamical systems via the pushforward of the conjugating map.
\begin{proposition}\label{prop:1}
If $S = h\circ T \circ h^{-1}$ and $\mu$ is $T$-invariant, then $h\# \mu$ is $S$-invariant. Moreover, if $\mu$ is $T$-ergodic, then $h\# \mu$ is additionally $S$-ergodic.
\end{proposition}
The full proof of Proposition \ref{prop:1} appears in Appendix \ref{appendix:lemmas}.
Proposition \ref{prop:1} motivates a key insight surrounding the issue of uniqueness when using invariant measures to compare dynamical systems. If $\mu\in \mathscr{P}(X)$ is invariant under both $T:X\to X$ and $S:X\to X$, then we cannot distinguish between the dynamical systems by studying the invariant measure alone. To resolve this challenge, we can instead make a change of coordinates and study the invariant measures of related, conjugate dynamical systems. That is, we will construct maps $h_T,h_S:X\to Y$, which are homeomorphic onto their image, and consider the invariant measures, $h_T\#\mu\in\mathscr{P}(Y)$ and $h_S\#\mu\in\mathscr{P}(Y)$, of the resulting conjugate systems, $\hat{T} = h_T \circ T \circ h_T^{-1}$ and $\hat{S} = h_S \circ S \circ h_S^{-1}$. Our ability to distinguish between $T$ and $S$ via $h_T\# \mu$ and $h_S\# \mu$ depends on how the conjugating maps $h_T$ and $h_S$ are chosen. In what follows, we argue that a powerful choice for our purpose is given by the delay-coordinate map, originating from Takens' seminal embedding theory~\cite{Takens1981DetectingSA}. 

\begin{definition}[Time-delay map]\label{def:time_delay}
    Consider a Polish space $X$, a map $T:X\to X$, an observable function $y:X\to \mathbb{R}$, and the time delay parameter $m\in\mathbb{N}$. The \textit{time-delay map} is defined as 
\begin{equation}\label{eq:delay_map}
    \Psi_{(y,T)}^{(m)}(x):=(y(x),y(T(x)),\dots, y(T^{m-1}(x)))\in \mathbb{R}^m,
    \end{equation}
    for each $x\in X$. 
\end{definition}
In \eqref{eq:delay_map}, we stress the dependence of the time-delay map $\Psi_{(y,T)}^{(m)}$ on the scalar observation function $y:X\to \mathbb{R}$, the underlying system $T:X\to X$, and the dimension $m\in\mathbb{N}$. 
When $T$ is continuous and the time-delay map $\Psi_{(y,T)}^{(m)}$ is injective, one can build a dynamical system in the reconstruction space $\mathbb{R}^m$ based on the following definition. 
\begin{definition}[Delay-coordinate dynamics]\label{def:delay_dynamics}
    Assume that $\Psi_{(y,T)}^{(m)}:X\to \mathbb{R}^m$, given in~\eqref{eq:delay_map}, is injective. Then, the \textit{delay-coordinate dynamics} are given by
    \begin{equation*}\label{eq:tdd}
\widehat{T}_{(y,m)}:\Psi_{(y,T)}^{(m)}(X)\to \Psi_{(y,T)}^{(m)}(X),\qquad \widehat{T}_{(y,m)}:= \Big[\Psi_{(y,T)}^{(m)}\Big] \circ T \circ \Big[\Psi_{(y,T)}^{(m)}\Big]^{-1}.
    \end{equation*}
\end{definition}

Since the delay-coordinate map $\widehat{T}_{(y,m)}$ is conjugate to the state-coordinate map $T:X\to X$, we now use Proposition~\ref{prop:1} to motivate our definition for an invariant measure in time-delay coordinates. When $T\#\mu =\mu$, the invariant measure in time-delay coordinates should be viewed as the corresponding invariant measure of the conjugate system $\widehat{T}_{(y,m)}$ given by the pushforward of $\mu$ under the time-delay map~\eqref{def:time_delay}; see Definition \ref{def:delay_im} below. Moreover, an illustration showing the difference between the state-coordinate invariant measure and the delay-coordinate invariant measure for the Lorenz-63 system is shown in Figure~\ref{fig:lorenz_visual}.

\begin{definition}[Invariant Measure in Time-Delay Coordinates]\label{def:delay_im}
      Assume that the delay map $\Psi_{(y,T)}^{(m)}:X\to \mathbb{R}^m$ in Definition~\ref{def:time_delay} is injective. Then, the probability measure 
      \begin{equation}\label{eq:tdim}
          \hat{\mu}_{(y,T)}^{(m)}:=\Psi_{(y,T)}^{(m)}\# \mu \in \mathcal{P}(\mathbb{R}^m)
      \end{equation} is the corresponding \textit{invariant measure in the time-delay coordinates.} We also refer to \eqref{eq:tdim} as a \textit{delay-coordinate invariant measure.} 
\end{definition}
\begin{figure}[h!]
    \centering
    \includegraphics[width=.8\linewidth]{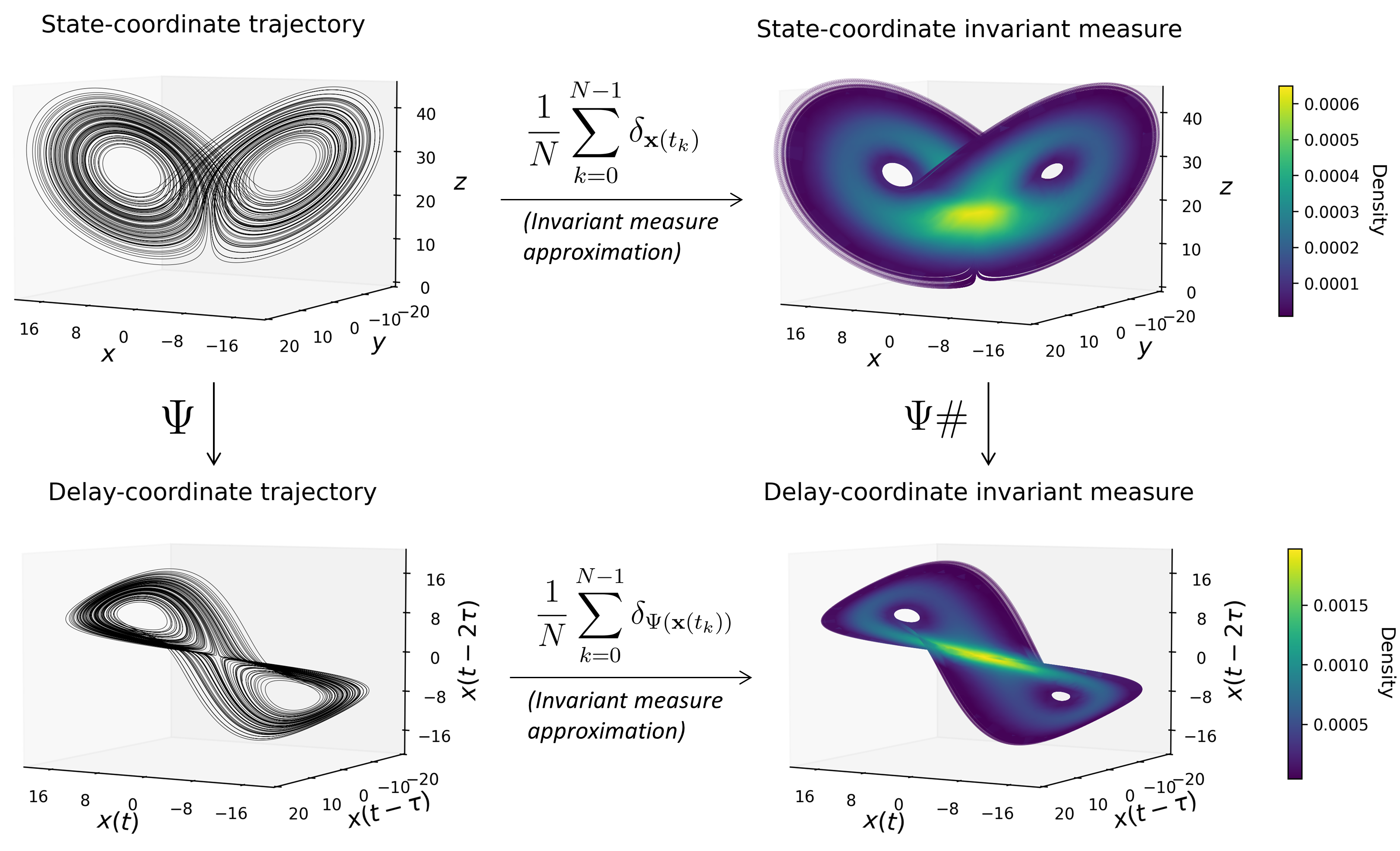}
    \caption{Visualization of the state-coordinate (top row) and delay-coordinate (bottom row) invariant measures for the Lorenz-63 system. Here, $\Psi$ is the delay-coordinate map. }
    \label{fig:lorenz_visual}
\end{figure}

We are now ready to state our main results of the section, which we formulate using the generalized version of Takens' theorem by Sauer, Yorke, and Casdagli; see Theorem \ref{thm:takens}. In what follows, we assume that $T,S:U\to U$ are diffeomorphisms of an open set $U \subseteq \mathbb{R}^n$, where $\mu\in\mathscr{P}(U)$ is $T$-invariant, $\nu\in\mathscr{P}(U)$ is $S$-invariant, and $\textup{supp}(\mu),\text{supp}(\nu) \subseteq U$ are compact. Moreover, we will assume that the embedding dimension $m\in\mathbb{N}$ has been chosen to satisfy $m > 2 \max\{d_{\mu},d_{\nu}\}$, where $d_{\mu} = \text{boxdim}(\text{supp}(\mu))$, is the essential dimension of $\text{supp}(\mu)$.   Lastly, we place a mild assumption on the growth of the number of periodic points of $T$ and $S$, which is standard in the literature of embedding theory; see Assumption \ref{assumption:1}.

Theorem \ref{thm:1} states that the equality of two invariant measures in time-delay coordinates implies that the underlying dynamical systems are topologically conjugate on the supports of their respective invariant measures. 
\begin{theorem}\label{thm:1}
      The equality   $\hat{\mu}_{(y,T)}^{(m+1)} = \hat{\nu}_{(y,S)}^{(m+1)}$ implies  $T|_{\textup{supp}(\mu)}$ and $S|_{\textup{supp}(\nu)}$ are topologically conjugate, for almost every $y\in C^1(U,\mathbb{R})$.
\end{theorem}
    The full proof of Theorem \ref{thm:1} appears in Appendix \ref{appendix:lemmas}. It is worth noting that the required embedding dimension to form the delay-coordinate invariant measures in Theorem~\ref{thm:1} is one dimension greater than the dimension in the generalized Takens' theorem in~\cite{sauer1991embedology}. Indeed, the key to our proof is the observation that a single point 
\begin{equation*}\label{eq:point_ex}
    (\lefteqn{\underbrace{\phantom{y(x),y(T(x)),\dots, y(T^{m-1}(x))}}_{\Psi_{(y,T)}^{(m)}(x)}}y(x),\overbrace{y(T(x)),\dots, y(T^{m-1}(x)), y(T^m(x))}^{\Psi_{(y,T)}^{(m)}(T(x))})\in\mathbb{R}^{m+1}
\end{equation*} in $m+1$ dimensional delay coordinates determines both $\Psi_{(y,T)}^{(m)}(x)$ and $\Psi_{(y,T)}^{(m)}(T(x))$. Together, these quantities account for one forward step of the dynamics in $m$-dimensional delay coordinates. 

For our next result, we consider the case when $\nu = \mu$, i.e., when both systems $T$ and $S$ share the same invariant measure. While Theorem \ref{thm:1} assumed access to only a single observable $y\in C^1(U,\mathbb{R})$, our second main result shows that the invariant measures in time-delay coordinates derived from a finite set of such observables, along with a suitable initial condition, uniquely determine the underlying system on $\text{supp}(\mu)$. 
 \begin{theorem}\label{thm:2}
    The conditions 
     \begin{enumerate}
         \item there exists $x^*\in B_{\mu,T}\cap \textup{supp}(\mu)$, such that $T^k(x^*) = S^k(x^*)$ for $1\leq k \leq m-1$, and 
   \item $\hat{\mu}_{(y_j,T)}^{(m+1)} = \hat{\mu}_{(y_j,S)}^{(m+1)}$ for $1\leq j \leq m$, where $Y:=(y_1,\dots, y_m)$ is a vector-valued observable,
     \end{enumerate}
imply  that $T = S$ everywhere on $\textup{supp}(\mu)$, for almost every $Y\in C^1(U,\mathbb{R}^m)$.
 \end{theorem}
 Our proof of Theorem \ref{thm:2}, which is presented in Appendix \ref{appendix:lemmas} uses Theorem \ref{thm:1}, along with a generalized Whitney embedding theorem \cite{sauer1991embedology}.

\subsection{Discussion}\label{subsec:discussions2}
Theorems~\ref{thm:1} and~\ref{thm:2} are our main results for comparing dynamical systems through delay-coordinate invariant measures, based upon uniform time-delay embeddings. We remark that our results hold in great generality, as we place no assumptions on the invariant measures under consideration. In particular, these statements still hold when the invariant measures are singular with respect to the Lebesgue measure and have fractal support, a common situation for attracting dynamical systems. Moreover, we have formulated our results using the mathematical theory of prevalence~\cite{hunt1992prevalence}, arguing that the conclusions of Theorems \ref{thm:1} and \ref{thm:2} hold for almost all observation functions. Practitioners often have little control over the measurement device used for data collection, so our proposed approach of comparing dynamical systems using invariant measures in time-delay coordinates remains broadly applicable.

In Theorem~\ref{thm:1}, we show that the equality of two invariant measures in time-delay coordinates implies a topological conjugacy of the underlying dynamical systems. This represents a significant advancement compared to standard state-coordinate invariant measure matching. One of the primary shortcomings of \cite{greve2019data,yang2023optimal}, which perform system identification by comparing state-coordinate invariant measures, is the inability of such approaches to enforce ergodicity of the reconstructed system. Since ergodicity is preserved under topological conjugacy (see Proposition~\ref{prop:1}), the approach of comparing invariant measures in time-delay coordinates, inspired by Theorem~\ref{thm:1}, resolves this challenge.

It is also worth noting that the invariant measures in delay coordinates used in Theorem \ref{thm:1} can be constructed entirely through partial observations of the full state. This property is essential in large-scale applications when the full state is not directly observable. However, in certain situations when one can probe the full state, Theorem \ref{thm:2} tells us that the delay-coordinate invariant measures corresponding to a finite set of observation functions contain sufficient information to reconstruct the underlying dynamics, provided that a suitable initial condition also holds. The condition $T^k(x^*) = S^k(x^*)$, for $1\leq k \leq m-1$, appearing in Theorem \ref{thm:2} is relatively mild, only requiring that the two systems agree for finitely many iterations at some initial condition $x^*\in \text{supp}(\mu)\cap B_{\mu,T}.$

\subsection{Numerical Results}\label{subsec:numerics3}

In Figure~\ref{fig:lorenz_compare} we compare the ability of loss functions based on (i) the state-coordinate invariant measure and (ii) the delay-coordinate invariant measure to reconstruct the full dynamics of the Lorenz-63 system.

\begin{figure}[h!]
    \centering
    \includegraphics[width=.8\linewidth]{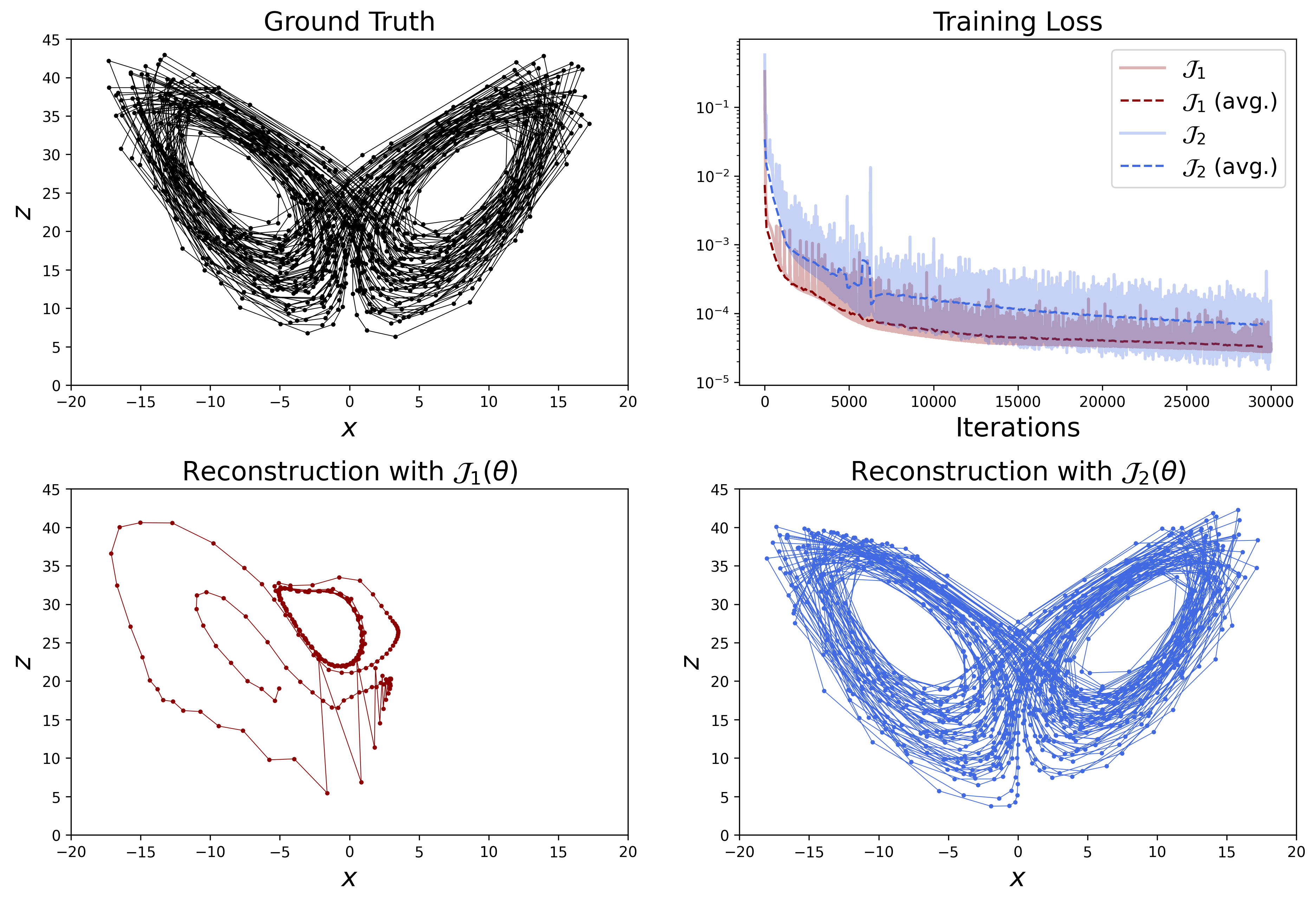}
    \caption{ While the loss $\mathcal{J}_1$, based only on the state-coordinate invariant measure, is insufficient for reconstructing the dynamics (bottom left), the loss $\mathcal{J}_2,$ which enforces equality of the delay-coordinate invariant measures, yields an accurate reconstruction (bottom right). Both models have identical initialization, architecture, and are trained using  the Adam optimizer \cite{kingma2014adam}. }
    \label{fig:lorenz_compare}
\end{figure}
In particular, we parameterize $T_{\theta}:\mathbb{R}^3\to \mathbb{R}^3$ as a fully-connected neural network, where $\theta\in\mathbb{R}^p$ comprises the network's weights and biases, and we attempt to recover the dynamics of the Lorenz-63 system using the loss functions 
$$\begin{cases}
    \mathcal{J}_1(\theta):=\mathcal{D}(T_{\theta}\#\mu^*,T^*\#\mu^*)\\
    \mathcal{J}_2(\theta):=\mathcal{D}(T_{\theta}\#\mu^*,T^*\#\mu^*)+\mathcal{D}(\Psi_{\theta}\#\mu^*,\Psi^*\#\mu^*),
\end{cases}$$ where $\mu^*\in\mathscr{P}(\mathbb{R}^3)$ is the Lorenz-63 system's invariant measure, and $\mathcal{D}:\mathscr{P}(\mathbb{R}^3)\times \mathscr{P}(\mathbb{R}^3)\to [0,\infty)$ is a metric or divergence on the space of probability measures, which we choose as the energy distance Maximum-Mean Discrepancy \cite{feydy2019interpolating}. Moreover, the map $T^*:\mathbb{R}^3\to \mathbb{R}^3$ represents the ground truth time-$\tau$ flow map of the Lorenz-63 system, and $\Psi^*$ is the ground truth time-delay map for $T^*$, based on the partially observed first component of the dynamics. The map $\Psi_{\theta}$ is the delay-coordinate map parameterized by the reconstructed dynamics $T_{\theta}$. Notably, if $\mathcal{J}_1(\theta)$ is reduced to zero, then $T_{\theta}$ admits $\mu^*$ as an invariant measure, and if $\mathcal{J}_2(\theta)$ is reduced to zero, then the delay-coordinate invariant measures for $T_{\theta}$ and $T^*$ additionally match. The results of Figure~\ref{fig:lorenz_compare} indicate that matching state-coordinate invariant measures is insufficient for reconstructing the Lorenz-63 dynamics, whereas matching the delay-coordinate invariant measures can yield a successful reconstruction.

\section{Conclusions}\label{sec:conclusions}

In this paper, we introduced a novel approach to performing dynamical system identification which involved matching simulated and observed physical invariant measures. In Section \ref{sec:PDE_IM}, we utilized the stationary solution of the Fokker--Planck equation as a surrogate model for the underlying system's invariant measure, which we numerically approximated as the dominant eigenvector of a Markov matrix originating from an upwind FVM discretization. The system identification problem was then recast as a large-scale PDE-constrained optimization procedure in which a metric or divergence on the space of probability measures compared the data-extracted invariant measure with the stationary Fokker--Planck surrogate model. While this framework exhibited compelling performance across several numerical tests, including a slowly sampled limit cycle (see Figure \ref{fig:initcompare}), a chaotic attractor (see Figure \ref{fig:lorenz}), and real-world HET data (see Figure \ref{fig:HET}), two outstanding challenges remained. 

First, the use of the FVM on a uniformly spaced mesh caused the approach in Section \ref{sec:PDE_IM} to become intractable as the dimension of the state increased. Given that the attractors of many dynamical systems are supported on low dimensional manifolds within high-dimensional state-spaces, in Section~\ref{sec:mesh} we instead constructed a data-adaptive unstructured mesh concentrated on the observed attractor. Moreover, rather than approximating the PFO using the FVM on these cells, we used a sample-based Galerkin approximation which is straightforward to evaluate even in high dimensions. We established convergence of this approximation to the underlying PFO in a suitable operator norm (see Theorem~\ref{thm:convergence}), and we also studied optimal strategies for constructing the data adaptive mesh which reduce the variance of our Monte--Carlo estimation (see Section~\ref{subsec:MC_analysis}). The benefits of the unstructured mesh approach compared to the uniform mesh were shown in Figure~\ref{fig:Cat}, and its applicability to high-dimensional systems in Figure \ref{fig:L96}.

The second challenge of the framework from Section \ref{sec:PDE_IM} was the lack of dynamical system identifiability from invariant measures. To this end, we introduced the time-delay coordinate transformation from Takens' seminal embedding theory as a crucial change of coordinates within which one obtains unique identificaiton from the invariant statistics alone. We proved that matching delay-coordinate invariant measures implies a topological conjugacy of the underlying dynamics (see Theorem \ref{thm:1}), and that the non-uniqueness from the conjugacy relation can be entirely eliminated by matching finitely many additional delay-coordinate invariant measures (see Theorem \ref{thm:2}). 

There are several interesting directions of future work to explore. Perhaps the most natural is to combine the computationally efficient data-adaptive mesh approach introduced in Section~\ref{sec:mesh} with the theoretical framework of matching delay-coordinate invariant measures introduced in Section~\ref{sec:delay_IM}. While Section \ref{sec:delay_IM} deals mostly with the issue of non-uniqueness from a theoretical perspective, leveraging the computational approach from Section~\ref{sec:delay_IM} could lead to a practical means for performing unique large-scale system identification using invariant measures in time-delay coordinates. Moreover, a different approach to guarantee uniqueness of the reconstructed dynamics could involve comparing additional eigenvectors of the Markov matrices, or simply the entire matrices as done in~\eqref{eq:markopt}, which may be useful when it is either challenging to determine suitable embedding parameters from data or the assumptions of Takens' theorem are not met. Lastly, while Section~\ref{sec:delay_IM} studies uniqueness of the inverse problem~\eqref{eq:opt2}, a separate theoretical treatment for stability is also needed.
\section*{Acknowledgements}
The author was supported by a fellowship award under contract FA9550-21-F-0003 through the National Defense Science and Engineering Graduate (NDSEG) Fellowship Program, sponsored by the Air Force Research Laboratory (AFRL), the Office of Naval Research (ONR) and the Army Research Office (ARO). The author would also like to thank Yunan Yang (Cornell University), Robert Martin (DEVCOM Army Research Laboratory), and Yinong Huang (University of Wisconsin-Madison) for their helpful discussions and insights.

\appendix
\addtocontents{toc}{\protect\setcounter{tocdepth}{1}}
\section{Appendix}
\subsection{Choice of Objective Function}\label{appendix:obj}
Here we summarize possible choices for $\mathcal{D}$, which is used as a metric or divergence on the space of probability measures in Section \ref{sec:PDE_IM}.
\begin{itemize}
  \item 
\textbf{Quadratic Wasserstein Distance:} For probability measures $\rho$ and $\rho^*$ on $\Omega$, with finite second-order moments, the squared quadratic Wasserstein distance is defined by
    \[
    W^2_2(\rho,\rho^*):=  \inf_{T_{\rho,\rho^*}\in \mathcal{P}}\int_{\Omega}|x-T_{\rho,\rho^*}(x)|^2 d\rho(x),
    \]
    where $$\mathcal{P}:=\{ T:\Omega \to \Omega: \rho(T^{-1}(B)) = \rho^*(B), \, B\in\mathscr{B}\}$$ is the set of maps that push $\rho$ forward into $\rho^*$. With an abuse of notation, we also use $\rho(x)$ and $\rho^*(x)$ to denote the densities of $\rho$ and $\rho^*$ respectively. 
\item \textbf{Squared $L^2$ Norm:} The squared $L^2$ distance between $\rho$ and $\rho^*$ is
\begin{align*}
    \mathcal{J} &=\frac{1}{2}\int_{\Omega} |\rho(x) - \rho^*(x)|^2dx\end{align*}
\item \textbf{KL-Divergence:} The KL-divergence between $\rho$ and $\rho^*$ is defined by
\begin{align*}
    \mathcal{J} = D_{\text{KL}}(\rho,\rho^*)&:= \int_{\Omega} \rho^*(x) \log\bigg( \frac{\rho^*(x)}{\rho(x)}\bigg)dx
\end{align*}
Based on definitions of the KL-divergence, it is clear that we may encounter numerical instability issues if either $\rho$ or $\rho^*$ is not supported on the entire domain $\Omega$. Thus, we remark that for the computation of both the KL and JS divergences, we restrict the domain $\Omega$ to regions where both $\rho$ and $\rho^*$ are strictly positive. 
\end{itemize}
\subsection{The Finite-Volume Matrix and its Derivative}\label{subsec:FVM_appendix}
The explicit form of the matrix $K_i$, defined in Section \ref{sec:FVM} is given by 
\begin{equation*} 
\tiny
K_{i} :=\begin{tikzpicture}[baseline={-0.5ex},mymatrixenv]
\matrix [mymatrix,inner sep=4pt](m)
{
\ddots  \\
& -v^{i,-}_{j-1} +\displaystyle \frac{D}{\Delta x_i }\\ 
\ddots& \vdots\\ & v^{i, -}_{j-1}-w^{i,+}_{j-1} -\displaystyle \frac{2D}{\Delta x_i }&  -v^{i,-}_{j} +\displaystyle \frac{D}{\Delta x_i } \\ 
\ddots & \vdots &\vdots \\ 
& w^{i,+}_{j-1}+\displaystyle \frac{D}{\Delta x_i }& v^{i,-}_{j} - w^{i,+}_{j} -\displaystyle \frac{2D}{\Delta x_i }& -v^{i,-}_{j+1}+\displaystyle \frac{D}{\Delta x_i }\\
& &  \vdots & \vdots & \ddots\\
& & w^{i,+}_{j}+\displaystyle \frac{D}{\Delta x_i } & v^{i,-}_{j+1}-w^{i,+}_{j+1}-2\displaystyle \frac{D}{\Delta x_i } \\
& & & \vdots & \ddots \\
& & &  w^{i,+}_{j+1}+\displaystyle \frac{D}{\Delta x_i } \\
& & & & \ddots \\
};
\mymatrixbraceright{1}{3}{$S_{i}$}
\end{tikzpicture}\in\mathbb{R}^{N\times N},
\end{equation*}
where the spacing between diagonals is $S_{i} := \prod_{j = 1}^{i-1}n_j$. In our adjoint-state gradient derivation (see Section \ref{sec:gradientcalc}), we take advantage of the particular structure of $K_i$ to evaluate
\begin{equation}\label{eq:M2}
\tiny
\frac{\partial K_i}{\partial v_j^i} = \begin{tikzpicture}[baseline={-0.5ex},mymatrixenv]
\matrix [mymatrix,inner sep=4pt](m)
{
\ddots \\& 0 \\ 
\ddots& \vdots\\ &  -H(v^{i}_{j}) & \hspace{.5cm}  -(1-H(v^{i}_{j}))  \\ 
\ddots & \vdots &\vdots \\ 
& H(v^{i}_{j})&\hspace{.5cm} (1-H(v^{i}_{j})) & \hspace{.5cm}0\\
& &  \hspace{.5cm}\vdots & \hspace{.5cm}\vdots & \hspace{.5cm}\ddots\\
& & \hspace{.5cm}0 &\hspace{.5cm}0 \\
& & & \hspace{.5cm}\vdots & \hspace{.5cm}\ddots \\
& & &  \hspace{.5cm}0 \\
& & & & \hspace{.5cm}\ddots \\
};
\mymatrixbraceright{1}{3}{$S_{i}$}
\end{tikzpicture} \normalsize \hspace{1cm} H(x ):= \begin{cases}
1, & x > 0\vspace{.2cm}\\
0, & x \leq 0
\end{cases},    
\end{equation}
where $H$ is the Heaviside function. Thus, $\partial_{ v_j^i}K_i$ can only be nonzero in the $(j,j)$, $(j,j-S_i)$, $(j-S_i,j)$, and $(j-S_i,j-S_i)$-th entries.
\subsection{Approximation of the Perron--Frobenius Operator via a Data-Driven Mesh: Full Proofs}\label{appendix:proj_proof}
We begin by establishing several lemmas which are necessary for our proof of Theorem \ref{thm:convergence}. Many of the following lemmas generalize results from \cite{li1976finite} from the setting of one-dimensional interval maps to arbitrary compact metric spaces. In what follows, we omit the superscript $\varepsilon$ from the operators introduced in Section \ref{sec:mesh} when considering the case $\psi_{j,n}^{(\varepsilon)} = \chi_{C_j}. $  We begin by checking that the projected Perron--Frobenius operator $\mathcal{L}^{(n)}$, which is clearly linear, is also positive and mass-preserving.
\begin{lemma}\label{lemma:g1}
If $f\in \Delta_n$ satisfies $f\geq 0$, then $\mathcal{L}^{(n)}f \geq 0$ as well and $\int_X \mathcal{L}^{(n)}f d\mu = \int_X f d\mu.$
\end{lemma}
\begin{proof}
 Write $f = \sum_{i=1}^n a_i \phi_{i,n}$, with $a_i \geq 0$ for $1\leq i \leq n$. Then, by linearity we have that 
\begin{equation}\label{eq:eq1}
    \mathcal{L}^{(n)}f  = \sum_{i=1}^n a_i \mathcal{L}^{(n)}\phi_{i,n} = \sum_{i=1}^n \sum_{j=1}^n a_i M_{i,j}^{(n)}\phi_{j,n}.
\end{equation}
Since $a_i,M_{i,j}^{(n)}\geq 0$ for each $1\leq i,j\leq n$, it holds that $\mathcal{P}^{(n)}f \geq 0$. Moreover, using the above expression for $\mathcal{P}^{(n)}f$, we have
\begin{align*}
    \int_X \mathcal{L}^{(n)}f d\mu =\sum_{i=1}^n \sum_{j=1}^n\bigg( a_i M_{i,j}^{(n)}\int_X\phi_{j,n}d\mu\bigg) = \sum_{i=1}^n a_i \sum_{j=1}^n M_{i,j}^{(n)} = \sum_{i=1}^n a_i = \int_X f d\mu,
\end{align*}
which completes the proof.
\end{proof}
We next establish that the projection $\mathcal{Q}^{(n)}$ is also positive and mass-preserving. 
\begin{lemma}\label{lemma:g2}
If $f\in L_{\mu}^1(X)$ satisfies $f\geq 0$, then $\mathcal{Q}^{(n)} f \geq 0$ and $\int_X \mathcal{Q}^{(n)} f d\mu = \int_X fd\mu$. Moreover, for any $h\in L_{\mu}^1(X)$ we have $\int_X |\mathcal{Q}^{(n)} h|d\mu \leq \int_X |h|d\mu$.
\end{lemma}
\begin{proof}
Assume that $f\in L_{\mu}^1(X) $ satisfies $f\geq 0.$ Then the positivity of $\mathcal{Q}^{(n)} f$ is clear from the definition \eqref{eq:proj}, as $\int_{C_{i,n}} f d\mu \geq 0$, for each $1\leq i \leq n$. We now verify the mass-conservation property. Indeed, notice that
    $$\int_X \mathcal{Q}^{(n)} f d\mu =\int_{X}\sum_{j=1}^n \bigg(\frac{1}{\mu(C_{j,n})} \int_{C_{j,n}} f(y) d\mu(y)\bigg)\chi_{C_{j,n}} d\mu(x) = \sum_{j=1}^n \int_{C_{j,n}}f(y)d\mu(y) = \int_X f d\mu.$$
   We now prove the last part of the lemma, stating that $\mathcal{Q}^{(n)}$ is a contraction. Indeed, when $h\not\geq 0$ we can decompose $h = h^+ - h^-$ where $h^+,h^-\geq 0$, $|h| = h^+ + h^-$, and apply linearity of $\mathcal{Q}^{(n)}$ to deduce 
    $$\int_X |\mathcal{Q}^{(n)} h| d\mu \leq \int_X \mathcal{Q}^{(n)} h^+d\mu + \int_X \mathcal{Q}^{(n)} h^- d\mu = \int_X h^+ d\mu + \int_X h^- d\mu = \int_X |h|d\mu.$$
    Above, we have also used the fact that $\mathcal{Q}^{(n)}$ is positive.
\end{proof}
The following lemma highlights the relationship between $\mathcal{P}^{(n)}$ and $\mathcal{P}$.
\begin{lemma}\label{lemma:g3}
    It holds that $\mathcal{P}^{(n)} =  \mathcal{Q}^{(n)}\mathcal{P}\mathcal{Q}^{(n)} .$
\end{lemma}
\begin{proof}
    It suffices to show that 
    \begin{equation}\label{eq:wts}
    \mathcal{L}^{(n)}f = \mathcal{Q}^{(n)}\mathcal{P} f,\qquad f = \sum_{i=1}^n a_i \chi_{C_{i,n}}.
    \end{equation}
    Towards this, it follows by \eqref{eq:proj} that 
    \begin{align*}
        \mathcal{Q}^{(n)} \mathcal{P} f = \sum_{j=1}^n \sum_{i=1}^n \bigg( \frac{a_i}{\mu(C_{j,n})}\int_{C_{j,n}} \mathcal{P}\chi_{C_{i,n}} d\mu\bigg) \chi_{C_{j,n}}= \sum_{j=1}^n \sum_{i=1}^n a_i M^{(n)}_{i,j}\chi_{C_{j,n}}= \mathcal{L}^{(n)}f,
    \end{align*}
    where the final equality is a result of equation \eqref{eq:eq1}.
 
\end{proof}

We now rely on the fact that $X$ is a compact metric space, writing $d(\cdot,\cdot)$ to denote the metric on $X$, in order to prove the convergence of  $\mathcal{Q}^{(n)}$; see Lemma \ref{lemma:g4}
\begin{lemma}\label{lemma:g4}
    For all $f\in L_{\mu}^1$, it holds that $\mathcal{Q}^{(n)}f \to f$ in $L_{\mu}^1(X)$. 
\end{lemma}
\begin{proof}
    Note that for any $\varepsilon > 0$, we can find some $g\in C(X)$ such that $\|g-f\|_{L_{\mu}^1} < \varepsilon/3$. Since $X$ is compact $g$ is uniformly continuous.  Thus, there exists $\delta > 0$ such that if $d(x,y)<\delta$, then $|g(x)-g(y)| < \varepsilon/3$. By the assumption that $\displaystyle\lim_{n\to\infty}\max_{1\leq i\leq n}(\textup{diam}(C_{i,n})) = 0$, we may choose $N\in\mathbb{N}$ such that for all $n\geq N$ it holds that $\textup{diam}(C_{i,n}) < \delta$, for each $1\leq i \leq n$. Thus, for all $n \geq N$ if $x,y\in C_{i,n}$, for some $1\leq i \leq n$, it necessarily holds that $|g(x)-g(y)|<\varepsilon/3$. Using this fact, we have for $n\geq N$ and $1\leq i \leq n$ that
    \begin{align*}
        \int_{C_{i,n}} |\mathcal{Q}^{(n)} g(x) - g(x)|d\mu(x) &=\int_{C_{i,n}}\bigg| \sum_{j=1}^n \bigg(\frac{1}{\mu(C_{j,n})} \int_{C_{j,n}} g(y) d\mu(y) \bigg) \chi_{C_{j,n}}(x) - g(x)\bigg| d\mu(x)\\
        &=\int_{C_{i,n}}\bigg| \frac{1}{\mu(C_{i,n})} \int_{C_{i,n}} g(y) d\mu(y)  - g(x)\bigg| d\mu(x)\\
        &\leq \int_{C_{i,n}}\frac{1}{\mu(C_{i,n})} \int_{C_{i,n}}| g(y)-g(x)| d\mu(y)  d\mu(x)\leq \frac{\varepsilon\mu(C_{i,n}) }{3}.
    \end{align*}
    It then follows that 
    $$\int_X |\mathcal{Q}^{(n)} g - g| d\mu = \sum_{i=1}^n \int_{C_{i,n}} |\mathcal{Q}^{(n)}g -g| d\mu< \frac{\varepsilon}{3} \sum_{i=1}^N \mu(C_{i,n}) = \frac{\varepsilon}{3}. $$
   Everything together, we now have 
    \begin{align*}
        \int_X|\mathcal{Q}^{(n)} f - f| d\mu &\leq \int_X |\mathcal{Q}^{(n)} f - \mathcal{Q}^{(n)} g| d\mu + \int_X |\mathcal{Q}^{(n)} g - g |d\mu + \int_X |f - g|d\mu \\
        &\leq 2 \int_X |f-g|d\mu + \int_X |\mathcal{Q}^{(n)} g - g|d\mu \leq \varepsilon,
    \end{align*}
    which concludes the proof.
\end{proof}
We now prove Theorem \ref{thm:g1}, which relies on Lemmas \ref{lemma:g1}, \ref{lemma:g2}, \ref{lemma:g3}, and \ref{lemma:g4}. Theorem \ref{thm:g1} should be viewed as the version of Theorem \ref{thm:convergence} without any regularization.
\begin{theorem}\label{thm:g1}
    $\mathcal{P}^{(n)}$ is linear, positive, and Markov. Moreover, $\displaystyle \lim_{n\to\infty} \|\mathcal{P}^{(n)}-\mathcal{P}\|_{L_{\mu}^1\to L_{\mu}^1} = 0.$
\end{theorem}
\begin{proof}[Proof of Theorem \ref{thm:g1}]
The fact that $\mathcal{P}^{(n)}$ is a Markov operator is a direct consequence of Lemma~\ref{lemma:g1} and Lemma~\ref{lemma:g2}, as the composition of mass-preserving operators is still mass-preserving, and the composition of positive operators is still positive. We now prove the convergence result. Let $f\in L_{\mu}^1(X)$ be fixed with $\|f\|_{L_{\mu}^1} = 1$, and notice that 
    \begin{align}
        \| (\mathcal{P}^{(n)} - \mathcal{P}) f \|_{L_{\mu}^1} &= \| \mathcal{Q}^{(n)} \mathcal{P} \mathcal{Q}^{(n)}f- \mathcal{P} f \|_{L_{\mu}^1} \label{eq:exp1}\\
        & \leq \| \mathcal{Q}^{(n)} \mathcal{P} \mathcal{Q}^{(n)}f - \mathcal{Q}^{(n)} \mathcal{P} f\|_{L_{\mu}^1} + \|\mathcal{Q}^{(n)} \mathcal{P} f- \mathcal{P} f \|_{L_{\mu}^1} \label{eq:exp3} \\
        & \leq \|\mathcal{P} \mathcal{Q}^{(n)} f  - \mathcal{P} f\|_{L_{\mu}^1} + \|\mathcal{Q}^{(n)} \mathcal{P} f- \mathcal{P} f \|_{L_{\mu}^1}\xrightarrow[]{n\to\infty} 0.  \label{eq:exp2} 
    \end{align}
    Above, \eqref{eq:exp1} follows from Lemma \ref{lemma:g3}. Moreover, \eqref{eq:exp2} follows from \eqref{eq:exp3}, as it was established that $\mathcal{Q}^{(n)}$ is a contraction on $L_{\mu}^1(X)$ in Lemma \ref{lemma:2}. The first term in \eqref{eq:exp2} then goes to zero as $\mathcal{P}$ is a bounded linear operator, due to the fact that it is Markov \cite[Proposition 3.1.1]{lasota2013chaos}, and hence continuous. Thus, since $\mathcal{Q}^{(n)} f \to f$ in $L_{\mu}^1(X)$ by Lemma \ref{lemma:g4}, it holds that $\mathcal{P}\mathcal{Q}^{(n)}f \to \mathcal{P}f$ in $L_{\mu}^1(X)$, as well. The second term in \eqref{eq:exp2} also goes to zero as a consequence of \ref{lemma:g4}. 
\end{proof}
We can now prove Theorem \ref{thm:convergence}, which establishes convergence of the regularized Galerkin projection appearing in Section \ref{sec:mesh}.
\begin{proof}[Proof of Theorem \ref{thm:convergence}]
    Lemma \ref{lemma:g2} shows that $\mathcal{Q}^{(n)}:L_{\mu}^1(X)\to \Delta_n$ is mass-preserving and positive. Thus, to show that $\mathcal{P}^{(n,\varepsilon)} = \mathcal{L}^{(n,\varepsilon)}\mathcal{Q}^{(n)}$ is Markov, it suffices to show that $\mathcal{P}^{(n,\varepsilon)}:\Delta_n\to \Delta_n$ is also mass-preserving and positive. We remark that positivity is clear from the definition \eqref{eq:ulam_projection_eps}, as $\phi_{i,n}(x), \psi_{i,n}^{(\varepsilon)}(x)\geq 0$ for all $x\in X$,  $n\in\mathbb{N}$, $i\leq n$, and $\varepsilon > 0$. Thus, we will focus on verifying the mass-preservation property. Towards this, let $f\in \Delta_n$ be fixed with $f\geq 0$ and write
$f = \sum_{i=1}^n a_i \phi_{i,n}$. Now, notice that 
\begin{align}
    \int_X \mathcal{P}^{(n,\varepsilon)}fd\mu = \sum_{i=1}^n a_i \bigg(\int_X \mathcal{P}^{(n,\varepsilon)} \phi_{i,n} d\mu\bigg)
    &= \sum_{i=1}^n a_i \sum_{j=1}^n\bigg( \int_X  M_{i,j}^{(n,\varepsilon)}\phi_{j,n}d\mu\bigg) \label{eq:l1}\\
    &= \sum_{i=1}^n a_i \sum_{j=1}^nM_{i,j}^{(n,\varepsilon)}\label{eq:l2}\\
    &=\sum_{i=1}^n a_i \bigg(\int_X \sum_{j=1}^n\psi^{(\varepsilon)}_{j,n}\circ T\phi_{i,n}d\mu\bigg) \nonumber\\
    &= \int_X f d\mu.\label{eq:l3}
\end{align}
Above, \eqref{eq:l1} follows from linearity, \eqref{eq:l2} uses the fact that $\phi_{j,n}$ has unit integral, and \eqref{eq:l3} leverages the fact that $\{\psi_{j,n}^{(\varepsilon)}\}_{j=1}^n$ forms a partition of unity. The rest of the equalities follow from simple rearrangements. Thus, we have established that $\hat{\mathcal{P}}^{(n,\varepsilon)}$ is Markov.

We now focus on proving the convergence. Let $n\in\mathbb{N}$ be fixed and assume $i,j\leq n$. Since $\psi_{j,n}^{(\varepsilon)} \xrightarrow[]{\varepsilon \to 0} \chi_{C_{j,n}}$ pointwise $\mu$-almost everywhere, we have that 
$$g_{\varepsilon}(x):= \Big| (\psi_{j,n}^{(\varepsilon)}\circ T)(x) - (\chi_{C_{j,n}}\circ T)(x) \Big| \phi_{i,n}(x)$$
satisfies $g_{\varepsilon}(x) \xrightarrow[]{\varepsilon \to 0}$ pointwise $\mu$-almost everywhere. Moreover, by construction $0\leq g_{\varepsilon}(x) \leq 1$ for all $\varepsilon > 0$ and all $x\in X$. Thus, it follows by Lebegue's dominated convergence theorem that
\begin{align*}
    |M_{i,j}^{(n,\varepsilon)} - M_{i,j}^{(n)}| \leq \int_X \Big| \psi_{j,n}^{(\varepsilon)}\circ T - \chi_{C_{j,n}}\circ T \Big| \phi_{i,n}d\mu=\int_X g_{\varepsilon}d\mu \xrightarrow[]{\varepsilon \to 0} 0.
\end{align*}
Using this result, notice that for all $x\in X$ it holds that
$$\lim_{\varepsilon \to 0}(\mathcal{L}^{(n,\varepsilon)} \phi_{i,n})(x) = \lim_{\varepsilon \to 0 } \sum_{j=1}^n M_{i,j}^{(n,\varepsilon)}\phi_{j,n}(x) = \sum_{j=1}^n M_{i,j}^{(n)}\phi_{j,n}(x) = \mathcal{L}^{(n)} \phi_{i,n}(x).$$
Moreover, since $0 \leq \phi_{i,n}\leq 1$ the dominated convergence theorem allows us to convert the pointwise convergence above into $L_{\mu}^1$ convergence. That is,
$$\lim_{\varepsilon \to 0}\|\mathcal{L}^{(n,\varepsilon)}\phi_{i,n} - \mathcal{L}^{(n)}\phi_{i,n}\|_{L_{\mu}^1} = 0.$$
The result then extends for arbitrary $f\in \Delta_n$. That is, if we write $f = \sum_{i=1}^na_i \phi_{i,n}$ we have by the triangle inequality that 
\begin{equation}\label{eq:eps_conv}
    \lim_{\varepsilon \to 0}\|\mathcal{L}^{(n,\varepsilon)}f - \mathcal{L}^{(n)}f\|_{L_{\mu}^1} \leq \lim_{\varepsilon \to 0} \sum_{i=1}^n |a_i| \|\mathcal{L}^{(n,\varepsilon)}\phi_{i,n} - \mathcal{L}^{(n)}\phi_{i,n}\|_{L_{\mu}^1} = 0.
\end{equation}
Finally, for any arbitrary $f\in L_{\mu}^1(X)$ with $\|f\|_{L_{\mu}^1} = 1$ we have that 
\begin{align*}
    \lim_{n\to \infty}\lim_{\varepsilon \to 0}\| \mathcal{P}^{(n,\varepsilon)} f - \mathcal{P} f\|_{L_{\mu}^1} &=  \lim_{n\to \infty}\lim_{\varepsilon \to 0}\| \mathcal{L}^{(n,\varepsilon)}\mathcal{Q}^{(n)}f - \mathcal{P} f\|_{L_{\mu}^1} \nonumber \\
    &\leq \lim_{n\to \infty}\lim_{\varepsilon \to 0}\| \mathcal{L}^{(n,\varepsilon)}\mathcal{Q}^{(n)}f  -\mathcal{L}^{(n)}\mathcal{Q}^{(n)}f\|_{L_{\mu}^1}+ \lim_{n\to \infty}\lim_{\varepsilon \to 0}\| \mathcal{P}^{(n)}f - \mathcal{P} f\|_{L_{\mu}^1} \nonumber\\
    & = \lim_{n\to \infty} \|\mathcal{P}^{(n)}f - \mathcal{P}f\|_{L_{\mu}^1}= 0.
\end{align*}
Above, we have used the convergence \eqref{eq:eps_conv} to move from the second line to the third line, as well as the convergence result from Theorem \ref{thm:g1} to obtain the final equality. 
\end{proof}
\subsection{Prevalence}\label{appendix:prev}

The definition of prevalence generalizes the notion of ``Lebesgue almost everywhere'' from finite-dimensional vector spaces to infinite dimensional function vector spaces~\cite{hunt1992prevalence,sauer1991embedology}. Towards this, we first define the Lebesgue measure on an arbitrary finite-dimensional vector space. 

\begin{definition}[Full Lebesgue measure on finite-dimensional vector spaces]\label{def:full_leb}
Let $E$ be a  $k$-dimensional vector space and consider a basis $\{v_1,\ldots,v_k\}\subseteq E$. A set $F\subseteq E$ has full Lebesgue measure if the coefficients $\big\{(a_1,\dots, a_k) \in \mathbb{R}^k: \sum_{i=1}^{k} a_iv_i \in F\big\}$ of its basis expansion have full Lebesgue measure in $\mathbb{R}^k$. 
\end{definition}

We will next define prevalence, which generalizes Definition \ref{def:full_leb} to infinite dimensional vector spaces. 
\begin{definition}[Prevalence]\label{def:prevalence}
    Let $V$ be a completely metrizable topological vector space. A Borel subset $S\subseteq V$ is said to be \textit{prevalent} if there is a finite-dimensional subspace $E\subseteq V$, known as a \textit{probe space}, such that for each $v\in V$, it holds that $v+e\in S$ for Lebesgue almost every $e\in E$. 
\end{definition}
Intuitively, prevalence means that almost all perturbations of an element $v\in V$ by elements of a probe space $E$ necessarily belong to the prevalent set $S$. Note that Definition 
 \ref{def:prevalence} reduces to Definition \ref{def:full_leb} if $V$ is a finite dimensional vector space. Moreover, if $V = C^k$, the space of functions whose $k$-th order derivative exists and is continuous, then it can be shown that a prevalent set is also dense in the $C^k$-topology \cite{hunt1992prevalence}. It also holds that the finite intersection of prevalent sets remains prevalent. 
\begin{lemma}\label{lemma:prevalence}
Let $S_1$ and $S_2$ be two prevalent subsets of a completely metrizable topological vector space $V$. Then, the intersection $S_1\cap S_2$ is also prevalent.
\end{lemma}

\subsection{Takens and Whitney Embedding Theorems}\label{appendix:embed}
It is often the case that dynamical trajectories are asymptotic to a compact attracting set $A$, which has a fractal structure and is not a manifold. The classical Takens and Whitney theorems require that such an attractor is contained within a smooth, compact manifold of dimension $d$ to embed the dynamics in $2d+1$-dimensional reconstruction space. The fractal dimension $d_A$ of the set $A$ might be much less than the manifold dimension, i.e., $d_A \ll d$. In such cases, it is desirable to consider more efficient approaches that can guarantee a system reconstruction in $2d_A+1$-dimensions. Towards this, we now recall the following definition of box-counting dimension, which will serve as our notion for the dimension of a fractal set.
 
 \begin{definition}[Box counting dimension]\label{def:box}
     Let $A\subseteq \mathbb{R}^n$ be a compact set. Its \textit{box counting dimension} is $$\text{boxdim}(A):=\displaystyle\lim_{\varepsilon \to 0} \frac{\log N(\varepsilon)}{\log(1/\varepsilon)},$$ where $N(\varepsilon)$ is the number of boxes with side-length $\varepsilon$ required to cover $A$. When the limit does not exist, one can define the upper and lower box-counting dimensions by replacing the limit with liminf and limsup, respectively.
 \end{definition}
 The following generalization of Whitney's embedding theorem comes from~\cite[Theorem 2.3]{sauer1991embedology}.
\begin{theorem}[Fractal Whitney Embedding]\label{thm:whitney}
    Let $A\subseteq \mathbb{R}^n$ be compact,  $d:= \textup{boxdim}(A)$, and $m > 2d$ be an integer. Then, almost every $F\in C^1(\mathbb{R}^n,\mathbb{R}^m)$ is injective on $A$.
\end{theorem}
There are certain assumptions required on the periodic points of a dynamical system for Takens' theorem to hold. Assumption \ref{assumption:1} makes these technical assumptions concise and easy to reference in our main results. In what follows, $DT^p$ denotes the derivative of the $p$-fold composition map $T^p$.
\begin{assumption}[Technical assumption on periodic points]\label{assumption:1}
  Let $T:U\to U$ be a diffeomorphism of an open set $U\subseteq \mathbb{R}^n$, $A\subseteq U$ be compact, and $m\in\mathbb{N}$. For each $p\leq m$:
  \begin{enumerate}
      \item The set $A_p\subseteq A$ of $p$-periodic points satisfies $\textup{boxdim}(A_p) < p/2$.
      \item The linearization $DT^p$ of each periodic orbit has distinct eigenvalues. 
  \end{enumerate}
  \end{assumption}
  
    
  
  We remark that when the diffeomorphism $T$ is given by the time-$\tau$ flow map of a Lipschitz continuous vector field, one can choose $\tau$ sufficiently small such that Assumption \ref{assumption:1} is satisfied; see \cite{sauer1991embedology}. The generalization of Takens' theorem can be found in \cite[Theorem 2.7]{sauer1991embedology} and is recalled in the main text; see Theorem \ref{thm:takens}.

 \subsection{Invariant Measures in Time-Delay Coordinates for Unique System Identification: Full Proofs}\label{appendix:lemmas}
\subsubsection{Proof of Theorem \ref{thm:1}}\label{subsec:proofs1}
In this section, we present a complete proof of Theorem \ref{thm:1}. We start by proving Proposition \ref{prop:1}
\begin{proof}[Proof of Proposition \ref{prop:1}]
    First, note that for any Borel subset $B\subseteq Y$ it holds that
    \begin{align*}
        (h\# \mu)(S^{-1}(B)) = \mu(h^{-1}(S^{-1}(B)) ) &= \mu (h^{-1}(h(T^{-1}(h^{-1}(B))))) \\
        &= \mu (T^{-1}(h^{-1}(B)))= \mu(h^{-1}(B))= (h\#\mu)(B),
    \end{align*}
    and thus $h\#\mu$ is $S$-invariant. Furthermore, if $S^{-1}(B) = B$, then it holds that $h(T^{-1}(h^{-1}(B))) = B,$ which means that $T^{-1}(h^{-1}(B)) = h^{-1}(B).$ If $\mu$ is ergodic, this directly implies that $\mu(h^{-1}(B)) \in \{0,1\}$, which gives us that $(h\#\mu)(B) \in \{0,1\}$, completing the proof.
\end{proof}

We next establish a useful lemma that relates the support of a probability measure to the support of its pushforward under a continuous mapping.
\begin{lemma}\label{lemma:1}
    Let $X$ and $Y$ be Polish spaces and $\mu\in \mathscr{P}(X)$ be a Borel probability measure with compact support, and assume that $f:X\to Y$ is continuous. Then, $f(\textup{supp}(\mu)) = \textup{supp}(f\# \mu).$ 
\end{lemma}
\begin{proof}
We prove the result via double inclusion, beginning first with the ``$\supseteq $" direction. Note that since $f$ is continuous and $\text{supp}(\mu)$ is compact in $X$, it holds that $D:=f(\text{supp}(\mu))$ is compact in $Y$. Moreover, using the fact that $\text{supp}(\mu) \subseteq f^{-1}(D)$, we obtain that 
    $$(f\# \mu)(D) = \mu(f^{-1}(D)) \geq \mu(\text{supp}(\mu)) = 1.$$
     Since $D$ is closed in $Y$, this implies that $\text{supp}(f\# \mu)\subseteq D = f(\text{supp}(\mu))$.

          To prove the ``$\subseteq$" inclusion, we note by continuity that $C:=f^{-1}(\text{supp}(f\# \mu))$ is a closed subset of $X$ which satisfies
$$\mu(C) = \mu(f^{-1}(\text{supp}(f\# \mu))) = (f\# \mu)(\text{supp}(f\# \mu)) = 1.$$
Thus, $\text{supp}(\mu)\subseteq C = f^{-1}(\text{supp}(f\#\mu)),$ which implies that 
\begin{equation*}\label{eq:supports}
   f(\text{supp}(\mu)) \subseteq f(f^{-1}(\text{supp}(f\# \mu)) \subseteq \text{supp}(f\#\mu). 
\end{equation*}
This establishes the inclusion $f(\text{supp}(\mu)) \subseteq \text{supp}(f\# \mu)$ and completes the proof. 
\end{proof}
The conclusion of Theorem \ref{thm:takens} states that the time-delay map $\Psi_{(y,T)}^{(m)}:U\to \mathbb{R}^m$ is injective on a compact subset $A\subseteq \mathbb{R}^n$. The following lemma will allow us to conclude that $\Psi_{(y,T)}^{(m)}:A\to \Psi_{(y,T)}^{(m)}(A)$ is in-fact a homeomorphism. This fact will be needed when we construct the conjugating map appearing in the conclusion of Theorem \ref{thm:1}.
\begin{lemma}{(\cite[Proposition~13.26]{sutherland2009introduction})}\label{lemma:2}
    Let $X$ and $Y$ be compact spaces and assume that $f:X\to Y$ is continuous and invertible. Then, $f$ is a homeomorphism, i.e., $f^{-1}$ is also continuous.
\end{lemma}

\begin{proof}[Proof of Theorem \ref{thm:1}]
 By Theorem \ref{thm:takens}, we have that
 \begin{align*}
     \mathcal{Y}_1&:=\{y\in C^1(U,\mathbb{R}):\Psi_{(y,S)}^{(m)}\text{ is injective on }\text{supp}(\nu)\}, \\
      \mathcal{Y}_2&:=\{y\in C^1(U,\mathbb{R}):\Psi_{(y,T)}^{(m)}\text{ is injective on }\text{supp}(\mu)\},
 \end{align*}
are prevalent subsets of $C^1(U,\mathbb{R})$. By Lemma~\ref{lemma:prevalence}, it holds that the intersection $\mathcal{Y}:=\mathcal{Y}_1\cap \mathcal{Y}_2$ is also prevalent in $C^1(U,\mathbb{R})$. Now, let $y\in \mathcal{Y}$ be fixed and assume that 
 \begin{equation}\label{eq:assumption}
     \Psi_{(y,S)}^{(m+1)}\# \nu = \Psi_{(y,T)}^{(m+1)}\#\mu.
 \end{equation}
    Since the mappings $$\Psi_{(y,S)}^{(m)}\Big|_{\text{supp}(\nu)}:\text{supp}(\nu)\to \Psi_{(y,S)}^{(m)}(\text{supp}(\nu)),\qquad \Psi_{(y,T)}^{(m)}\Big|_{\text{supp}(\mu)}:\text{supp}(\mu)\to \Psi_{(y,T)}^{(m)}(\text{supp}(\mu))$$ 
    are continuous, invertible, and the sets $\text{supp}(\nu)$ and $\text{supp}(\mu)$ are compact, it follows from Lemma~\ref{lemma:2} that the map $  \Theta_y:\text{supp}(\nu) \to \text{supp}(\mu)$, given by
    \begin{equation}\label{eq:theta}
    \Theta_y(x) :=\Bigg( \Big[\Psi_{(y,T)}^{(m)}\Big]\bigg|_{\text{supp}(\mu)}^{-1} \circ \Big[\Psi_{(y,S)}^{(m)}\Big]\bigg|_{\text{supp}(\nu)}\Bigg)(x),\qquad \forall x\in \text{supp}(\nu),
    \end{equation}
    is a well-defined homeomorphism. We now aim to show that $T|_{\text{supp}(\mu)}$ and $S|_{\text{supp}(\nu)}$ are topologically conjugate via the homeomorphism $\Theta_y$. 
    
    Returning to analyzing \eqref{eq:assumption}, we have by Lemma \ref{lemma:1} that 
    \begin{equation}\label{eq:sup_eq}
       \Psi_{(y,S)}^{(m+1)}(\text{supp}(\nu)) = \text{supp}\Big(\Psi_{(y,S)}^{(m+1)}\#\nu\Big) = \text{supp}\Big(\Psi_{(y,T)}^{(m+1)}\#\mu\Big) = \Psi_{(y,T)}^{(m+1)}(\text{supp}(\mu)). 
    \end{equation}
    Now, let $x\in \text{supp}(\nu)$ be fixed, and note that by the equality of sets \eqref{eq:sup_eq} and the definition of the time-delay map, there must exist some $z\in \text{supp}(\mu)$, such that
    \begin{equation}\label{eq:point_eq}    (\lefteqn{\underbrace{\phantom{y(x),y(S(x)),\dots, y(S^{m-1}(x))}}_{\Psi_{(y,S)}^{(m)}(x)}}y(x),\overbrace{y(S(x)),\dots, y(S^{m-1}(x)), y(S^m(x))}^{\Psi_{(y,S)}^{(m)}(S(x))}) = (\lefteqn{\underbrace{\phantom{y(x),y(S(x)),\dots, y(S^{m-1}(x))}}_{\Psi_{(y,T)}^{(m)}(z)}}y(z),\overbrace{y(T(z)),\dots, y(T^{m-1}(z)), y(T^m(z))}^{\Psi_{(y,T)}^{(m)}(T(z))})\,.
    \end{equation}
    By equating the first $m$ components, and then the last $m$ components, of the $m+1$ dimensional vectors appearing in \eqref{eq:point_eq}, we obtain the following two equalities:
    \begin{align}
        \Psi_{(y,S)}^{(m)}(x) &= \Psi_{(y,T)}^{(m)}(z) \label{eq:start},\\
        \Psi_{(y,S)}^{(m)}(S(x))&=\Psi_{(y,T)}^{(m)}(T(z))\label{eq:end}.
    \end{align}
    Since $x\in \text{supp}(\nu)$ and $z\in \text{supp}(\mu)$, we deduce from \eqref{eq:start} that $z = \Theta_y(x)$. Substituting this equality into~\eqref{eq:end} yields
    \begin{equation}\label{eq:next}
          \Psi_{(y,S)}^{(m)}(S(x))=\Psi_{(y,T)}^{(m)}(T(\Theta_y(x)))
    \end{equation}

    Moreover, since $T\# \mu = \mu$ and $S\# \nu = \nu$, it follows by Lemma \ref{lemma:1} that $T(\text{supp}(\mu)) = \text{supp}(\mu)$ and $S(\text{supp}(\nu)) = \text{supp}(\nu)$. Thus, $S(x)\in \text{supp}(\nu)$ and $T(\Theta_y(x))\in \text{supp}(\mu)$. This allows us to rearrange~\eqref{eq:next} to find 
    $$S(x) = (\Theta_y^{-1}\circ T \circ \Theta_y)(x).$$
    Since $x\in \text{supp}(\nu)$ was arbitrary, we have the general equality 
    $$S|_{\text{supp}(\nu)} = \Theta_y^{-1} \circ T|_{\text{supp}(\mu)} \circ \Theta_y,$$
   which completes the proof.
    \end{proof}
\subsubsection{Proof of Theorem \ref{thm:2}}\label{subsec:proofs2}
This section contains a proof of Theorem \ref{thm:2}. We begin by establishing several lemmas which are needed in our proof of the result. First, we will consider the case when two systems agree along an orbit, i.e., $S^k(x^*) = T^k(x^*)$ for all $k\in \mathbb{N}$. We will show that when $x^*\in B_{\mu,T}$ that the equality of $S$ and $T$ along the orbit initiated at $x^*$ implies that $S$ and $T$ agree everywhere on $\text{supp}(\mu)$. 
\begin{lemma}\label{lemma:7}
    Let $S,T:U\to U$ be continuous maps on an open set $U\subseteq \mathbb{R}^n$, let $\mu\in \mathscr{P}(U)$, and let $\textup{supp}(\mu)\subseteq U$ be compact. If for some $x^*\in B_{\mu,T}$ it holds that $S^k(x^*) = T^k(x^*)$ for all $k\in \mathbb{N}$, then $S|_{\textup{supp}(\mu)} = T|_{\textup{supp}(\mu)}$.
\end{lemma}
\begin{proof}
 Define the continuous function $\phi \in C(U)$ by setting $\phi(x):=\|S(x)-T(x)\|_2\geq 0$, for each $x\in U.$ 
Since $x^*\in B_{\mu,T}$, it follows from Definition \ref{def:basin} that 
\begin{align*}
    \int_{U}\|S(x) - T(x)\|_2d\mu(x) =\lim_{N\to \infty}\frac{1}{N}\sum_{k=0}^{N-1} \phi(T^k(x^*)).
\end{align*}
Moreover, since $T^k(x^*)=S^k(x^*)$ for all $k\in\mathbb{N}$, we find
 \begin{align*}
 \lim_{N\to \infty}\frac{1}{N}\sum_{k=0}^{N-1} \phi(T^k(x^*)) &=\lim_{N\to \infty}\frac{1}{N}\sum_{k=0}^{N-1} \| S(T^k(x^*))-T^{k+1}(x^*)\|_2\\& =\lim_{N\to \infty}\frac{1}{N}\sum_{k=0}^{N-1} \| S^{k+1}(x^*))-T^{k+1}(x^*)\|_2= 0.
 \end{align*}
 Since $\|S(x) - T(x)\|_2 \geq 0$ for all $x\in U$, it must hold that $S(x) = T(x)$ for $\mu$-almost all $x\in U$.  
 
 We now define $C:=\{x\in U:S(x) = T(x)\}$ and note $C = \{x\in U: S(x) - T(x) = 0\} = (S-T)^{-1}(\{0\}).$ Since $S-T$ is continuous, it holds that $C$ is closed. Moreover, since $\mu(C) = 1$ it follows by Definition \ref{def:support} that $\text{supp}(\mu)\subseteq C$. Thus, $S|_{\textup{supp}(\mu)} = T|_{\textup{supp}(\mu)},$ as wanted.
\end{proof}
In our proof of Theorem \ref{thm:2}, we will require that $\Psi_{(y_i,T)}^{(m)}$ is injective on $\text{supp}(\mu)$ for each $1\leq i \leq m$, where $Y = (y_1,\dots, y_m)\in C^1(U,\mathbb{R}^m).$ The following result (Lemma \ref{lemma:augment}) is used to show that the set of all $Y\in C^1(U,\mathbb{R}^m)$ for which this property holds is prevalent.
\begin{lemma}\label{lemma:augment}
 Assume that $\mathcal{Y}\subseteq C^1(U,\mathbb{R})$ is prevalent. Then, the set of functions $\bm{\mathcal{Y}}:=\{(y_1,\dots, y_m): y_i \in \mathcal{Y}\}$ is prevalent in $C^1(U,\mathbb{R}^m).$   
\end{lemma}
\begin{proof}
We first show that $\bm{\mathcal{Y}}\subseteq C^1(U,\mathbb{R}^m)$ is Borel. Towards this, define the projection $$\bm{\pi}_i: C^1(U,\mathbb{R}^m)\to C^1(U,\mathbb{R}), \qquad \bm{\pi}_i((y_1,\dots, y_m)):= y_i\in C^1(U,\mathbb{R}),$$
for all $(y_1,\dots, y_m)\in C^1(U,\mathbb{R}^m)$ and each $1\leq i \leq m.$. The projection $\bm{\pi}_i$ is continuous, hence Borel measurable. Since $\mathcal{Y}$ is Borel, it holds that $\bm{\pi}_i^{-1}(\mathcal{Y})$ is also Borel for $1\leq i \leq m$. Then, we can write $\bm{\mathcal{Y}}=\bigcap_{i=1}^m \bm{\pi}_i^{-1}(\mathcal{Y})$, which verifies that $\bm{\mathcal{Y}}\subseteq C^1(U,\mathbb{R}^m)$ is Borel. 

Since $\mathcal{Y}$ is prevalent, there exists a $k$-dimensional probe space $E\subseteq C^1(U,\mathbb{R})$ admitting a basis $\{e_i:1\leq i \leq k\}$, such that for any $y\in C^1(U,\mathbb{R})$, it holds that $y+ \sum_{i=1}^k a_ie_i \in \mathcal{Y}$ for Lebesgue-almost every $(a_1,\dots, a_k)\in \mathbb{R}^k$. Next, we will define the augmented probe space $$\bm{E}:=\{(v,\dots, v): v\in E\}\subseteq C^1(U,\mathbb{R}^m),$$ and note that $\{(e_i,\dots, e_i):1\leq i \leq k\}$ constitutes a basis for $\bm{E}.$  Now, fix $(y_1,\dots, y_m)\in C^1(U,\mathbb{R}^m)$ and observe that for each $1\leq j \leq m$, there exists a full Lebesgue measure set $B_j\subseteq \mathbb{R}^k$, such that $y_j+\sum_{i=1}^k a_ie_i \in \mathcal{Y}$ for all $(a_1,\dots, a_k)\in B_j$. By construction, the set $B:=\bigcap_{j=1}^k B_j$ has full Lebesgue measure in $\mathbb{R}^k$, and for all $(a_1,\dots,a_k)\in B$ it then holds that 
$$(y_1,\dots,y_m)+ \sum_{i=1}^{k}a_i(e_i,\dots, e_i) = \Bigg(y_1 + \sum_{i=1}^k a_i e_i,\dots, y_m + \sum_{i=1}^k a_i e_i\Bigg) \in \bm{\mathcal{Y}}, $$
which completes the proof.

\end{proof}

In our proof of Theorem \ref{thm:2}, we will utilize the generalized Takens theorem (Theorem \ref{thm:takens}) to conclude that each $\Psi_{(y_i,T)}^{(m)}$ is injective on the compact set $\text{supp}(\mu)$, for each $1\leq i \leq m$. We would also like to use the generalized Whitney theorem (Theorem~\ref{thm:whitney}) to conclude that $Y=(y_1,\dots, y_m)$ is injective on $\text{supp}(\mu)$. However, in the generalized Whitney theorem it is assumed that $Y\in C^1(\mathbb{R}^n,\mathbb{R}^m)$, whereas in our statement of Theorem~\ref{thm:2} we have $Y\in C^1(U,\mathbb{R}^m)$, for an arbitrary open set $U \supseteq \text{supp}(\mu)$. The following lemma leverages the Whitney extension theorem (see \cite{whitney1992analytic}) to provide a reformulation of the generalized Whitney embedding theorem in this setting.
\begin{lemma}\label{lemma:extend}
    Let $A\subseteq U \subseteq \mathbb{R}^n,$ where $A$ is compact and $U$ is open, set $d:=\textnormal{boxdim}(A)$, and let $m > 2d$ be an integer. For almost every smooth map $F\in C^1(U,\mathbb{R}^m)$, it holds that $F$ is one-to-one on $A$ and an immersion on each compact subset of a smooth manifold contained in $A$. 
\end{lemma}
\begin{proof}
    By Theorem \ref{thm:whitney}, the result holds when $U = \mathbb{R}^n$, and by \cite{sauer1991embedology} a suitable probe is given by space of linear maps between $\mathbb{R}^n$ and $\mathbb{R}^m$. Hereafter, we will denote this space by $\mathcal{L}(\mathbb{R}^n,\mathbb{R}^m)$. In the case when $U \neq \mathbb{R}^n$, we claim that the restricted space of linear maps $E:=\{ L|_U: L\in \mathcal{L}(\mathbb{R}^n,\mathbb{R}^m)\} $ is a suitable probe. We remark that if $W:=\{ L_1,\dots, L_{mn}\}$ is a basis for $\mathcal{L}(\mathbb{R}^n,\mathbb{R}^m)$, which is $nm$-dimensional, then $\hat{W}:=\{L_1|_U,\dots, L_{mn}|_U\}$ is a basis for $E$, which remains $nm$-dimensional. 
    
    To verify that $\hat{W}$ is a basis, let  $\hat{L}\in E$, and note that $\hat{L} = L|_U$ for some $L\in \mathcal{L}(\mathbb{R}^n,\mathbb{R}^m)$. Since $W$ is a basis for $\mathcal{L}(\mathbb{R}^n,\mathbb{R}^m)$, we can write $L = \sum_{i=1}^{nm} a_i L_i$ for some coefficients $a_i\in\mathbb{R}$, $1\leq i \leq nm$, and thus $\hat{L} = \sum_{i=1}^{nm}a_iL_i|_U.$ Therefore, $\hat{W}$ spans $E$. To see that the elements of $\hat{W}$ are linearly dependent, assume that $\sum_{i=1}^{nm}a_iL_i|_U \equiv 0 \in \mathbb{R}^m $ for some coefficients $a_i\in\mathbb{R}$ with $1\leq i \leq nm$. By linearity and the fact that $U$ is open, this implies $\sum_{i=1}^{nm}a_iL_i \equiv 0 \in \mathbb{R}^m $, and since $\{L_i\}_{i=1}^{nm}$ forms a basis, this implies $a_i = 0$ for each $1\leq i \leq nm$.


    To complete the proof, it remains to show that $E$ is a probe space. Towards this, let $F\in C^1(U,\mathbb{R}^m)$ be fixed, and note that by the Whitney Extension Theorem there exists $\tilde{F}\in C^1(\mathbb{R}^n,\mathbb{R}^m)$, such that $\tilde{F}|_A = F|_A$; see \cite{whitney1992analytic}. Then, by Theorem \ref{thm:whitney}, it holds that $\tilde{F}+\sum_{i=1}^{nm}a_i L_i$ is injective on $A$ for Lebesgue almost all $(a_1,\dots, a_{nm})\in \mathbb{R}^{nm}$. Since $\tilde{F}|_A = F|_A$, it follows that $F+\sum_{i=1}^{nm}a_iL_{i}|_U$ is injective on $A$ for Lebesgue almost all $(a_1,\dots, a_{nm})\in \mathbb{R}^{nm}$, which completes the proof.
\end{proof}

We are now ready to present the proof of Theorem \ref{thm:2}. 
\begin{proof}[Proof of Theorem \ref{thm:2}]
By Theorem \ref{thm:1}, the set of $y \in C^1(U,\mathbb{R})$ such that the equality of measures $\hat{\mu}_{(y,S)}^{(m+1)}=\hat{\mu}_{(y,T)}^{(m+1)}$ implies the topological conjugacy of $T|_{\text{supp}(\mu)}$ and $S|_{\text{supp}(\mu)}$ is prevalent. We denote this set by $\mathcal{Y}$. Moreover, by Lemma~\ref{lemma:augment} it holds that the set $\bm{\mathcal{Y}}:=\{(y_1,\dots, y_m): y_i \in \mathcal{Y}\}$ is prevalent in $C^1(U,\mathbb{R}^m).$ By Lemma~\ref{lemma:extend}, it follows that the set $\bm{\mathcal{Z}}\subseteq C^1(U,\mathbb{R}^m)$ of smooth maps that are injective on $\text{supp}(\mu)$ is also prevalent. 
 Since the finite intersection of prevalent sets is prevalent (see Lemma \ref{lemma:prevalence}) it holds that $\bm{\mathcal{W}}:=\bm{\mathcal{Y}}\cap \bm{\mathcal{Z}}$ is a prevalent subset of $C^1(U,\mathbb{R}^m)$. 

We now fix an  element $Y = (y_1,\dots, y_m)\in \bm{\mathcal{W}}$. Note that by the conclusion of Theorem \ref{thm:1}, we have 
    \begin{equation}\label{eq:supps}
        S|_{\textup{supp}(\mu)} = (\Theta_{y_i}^{-1} \circ T|_{\text{supp}(\mu)}\circ \Theta_{y_i}),\qquad 1\leq i \leq m,
    \end{equation}
    where $\Theta_{y_i}:\text{supp}(\mu)\to \text{supp}(\mu)$ is defined in \eqref{eq:theta}. Evaluating \eqref{eq:supps} at $x^*\in B_{\mu,T}\cap \text{supp}(\mu)$ and composing on both sides then yields 
        \begin{equation*}\label{eq:supps2}
        S|_{\textup{supp}(\mu)}^k(x^*) = (\Theta_{y_i}^{-1} \circ T|_{\text{supp}(\mu)}^k\circ \Theta_{y_i})(x^*),\qquad k\in \mathbb{N},\qquad 1\leq i \leq m.
    \end{equation*}
    Then, since $S^k(x^*) = T^k(x^*)$ for $0\leq k \leq m-1$, it holds that $\Theta_{y_i}(x^*) = x^*$ for $1\leq i \leq m$, which is a consequence of the definition of the delay map~\eqref{eq:delay_map} and the construction of $\Theta_{y_i}$; see \eqref{eq:theta}.  
    Therefore,
    \begin{equation}\label{eq:supps3}
        S|_{\textup{supp}(\mu)}^k(x^*) = (\Theta_{y_i} ^{-1}\circ T|_{\text{supp}(\mu)}^k)(x^*),\qquad k\in \mathbb{N},\qquad 1\leq i \leq m.
    \end{equation}
    Using the definition of $\Theta_{y_i}$, we now rearrange \eqref{eq:supps3} to find that 
     \begin{equation}\label{eq:supps4}
        \Psi_{(y_i,S)}^{(m)}(S|_{\textup{supp}(\mu)}^k(x^*)) = \Psi_{(y_i,T)}^{(m)}(T|_{\textup{supp}(\mu)}^k(x^*)),\qquad k\in \mathbb{N},\qquad 1\leq i \leq m.
    \end{equation}
    Again, using the definition of the delay map and equating the first components of the vectors in \eqref{eq:supps4} reveals that $y_i(S^k(x^*)) = y_i(T^k(x^*))$ for all $i=1,\ldots, m$ and any $k\in\mathbb{N}.$ Recall that $Y 
 = (y_1,\ldots,y_m)\in \bm{\mathcal{W}} = \bm{\mathcal{Y}} \cap \bm{\mathcal{Z}}$ is injective.  
 Thus, $S^k(x^*) = T^k(x^*)$ for all $k\in \mathbb{N}$. The conclusion then follows from Lemma~\ref{lemma:7}, which completes the proof.
\end{proof}

 \end{document}